\documentclass[a4paper,10pt,leqno]{amsart}
\usepackage[T1]{fontenc}
\usepackage[utf8]{inputenc}
\usepackage[margin=1.2in]{geometry}
\usepackage{lmodern}
\usepackage[english]{babel}
\usepackage{graphicx}
\usepackage{amsmath}
\usepackage{amsfonts}
\usepackage{amssymb} 
\usepackage{amsthm}
\usepackage{bbm}
\usepackage{bm} 
\usepackage{lipsum}
\usepackage{xcolor}
\usepackage{url}
\theoremstyle{plain}
\newtheorem{thm}{Theorem}
\newtheorem{lem}{Lemma}
\newtheorem{prop}{Proposition}

\newtheorem{defn}{Definition}
\newtheorem{exmp}{Example}

\newtheorem{rem}{Remark}

\def \be {\begin{equation}}
\def \ee {\end{equation}}

\begin{document}
\title[Finite state MFG and the Master Equation]{Convergence, Fluctuations and Large Deviations for Finite State Mean Field Games via the Master Equation}
\author{Alekos Cecchin}
\address[A. Cecchin and G. Pelino]
{\newline \indent Department of Mathematics ``Tullio Levi Civita''
\newline 
\indent University of Padua \newline
\indent Via Trieste 63, 35121 Padova, Italy
\newline }
\email[A. Cecchin]{alekos.cecchin@math.unipd.it}
\author{Guglielmo Pelino}
\email[G. Pelino]{guglielmo.pelino@math.unipd.it}
\thanks{The authors are supported by the PhD programme in Mathematical Sciences, Department of Mathematics, 
University of Padua (Italy), Progetto Dottorati - Fondazione Cassa di Risparmio di Padova e Rovigo and by the research project ``Nonlinear Partial Differential Equations: Asymptotic Problems and Mean-Field
Games'' of the Fondazione CaRiPaRo.
\\ We are very grateful to our PhD supervisors Paolo Dai Pra and Markus Fischer. We also thank Martino Bardi, Marco Cirant and Diogo Gomes for helpful discussions. 
}

\date{February 9, 2018}
\subjclass{60F05, 60F10, 60J27, 60K35, 91A10, 93E20} %
\keywords{Mean field games, finite state space, jump Markov processes, $N$-person games, Nash equilibrium, Master Equation, Propagation of chaos, Central Limit Theorem, Large Deviation Principle}

\begin{abstract}
We show the convergence of finite state symmetric $N$-player differential games, where players control their transition rates from state to state,  to a limiting dynamics given by a finite state Mean Field Game system made of two coupled forward-backward ODEs.  We exploit the so-called Master Equation, which in this finite-dimensional framework is a first order PDE in the simplex of probability measures, obtaining the convergence of the feedback Nash equilibria, the value functions and the optimal trajectories. The convergence argument requires only the regularity of a solution to the Master Equation. Moreover, we employ the convergence results to prove a Central Limit Theorem and a Large Deviation Principle for the evolution of the $N$-player empirical measures.
The well-posedness and regularity of solution to the Master Equation are also studied, under monotonicity assumptions.
\end{abstract}

\maketitle

\section{Introduction}

Mean Field Games were introduced independently by Lasry and Lions \cite{lasry} and by Huang et al.\cite{huang} as limit models for symmetric non-zero-sum non-cooperative $N$-player dynamic games  when the number $N$ of players tends to infinity.
For an introduction to the topic see for instance \cite{card1}, \cite{carm} or \cite{ben}, where the latter two deal also with mean-field type optimal control.
While a wide range of different classes of Mean Field Games has been considered up to now, here we focus on finite time horizon problems with continuous time dynamics under fully symmetric cost structure and complete information, where the position of each agent belongs to a finite state space.
In this setting, Mean Field Games were first analyzed in \cite{gomes_mohr} in discrete time, and then in \cite{gomes} in continuous time.  


The relation between the $N$-player game and its limit can be studied in two opposite directions: approximation and convergence. The approximation argument consists in proving that the solutions to the Mean Field Game allow to construct approximate Nash equilibria for the prelimit game, where the error in the approximation tends to zero as $N$ goes to infinity. Convergence goes in the opposite direction: are Nash equilibria for the $N$-player game converging to solutions of the Mean Field Game when the number of players tends to infinity?

Results in the first direction are much more common and easier to obtain: for the diffusive case without jumps see for instance \cite{huang}, \cite{carm}, \cite{carm1} and \cite{ben1}.  In the finite state space setting, this was achieved in  \cite{kolo} studying the infinitesimal generator, while in  \cite{alekos} an approximation result was found through a fully probabilistic approach, which allowed for less restrictive assumptions on the dynamics and the optimization costs.
On the other hand, results on convergence are fewer and even more recent: while the limits of $N$-player Nash equilibria in stochastic open-loop strategies can be completely characterized (see \cite{lacker} and \cite{fischer} for the diffusion case), the convergence problem for Nash equilibria in feedback form with full information is more difficult. A result in this direction is given by \cite{gomes} in our finite state setting, via the infinitesimal generator, but only if the time horizon is small. 

A breakthrough was achieved by Cardaliaguet et al. \cite{card} through the use of the so-called Master Equation, again in the diffusion case and also in the presence of a common noise. Their convergence argument relies on having a regular solution to the Master Equation, which in the diffusion case  - as originally remarked by Lions in \cite{lezioni} - is a kind of infinite dimensional transport equation on the space of probability measures. Its solution yields a solution of the Mean Field Game system for any initial time and initial distribution. Moreover, such system can be seen as the characteristics curves for the Master Equation; we will return to this shortly. 
If the Master Equation admits a unique regular solution, then this solution can be used to prove the convergence. The crucial ingredient in the proof consists in a coupling argument, in a similar fashion to the propagation of chaos property for uncontrolled systems (see \cite{sni}). Such coupling, in which independent copies of the limit process are compared to their prelimit counterparts, ultimately allows to get the desired convergence of the value functions of the $N$-player game to the solution to the Master Equation, as well as a form of convergence for the associated optimal feedback strategies and a propagation of chaos result for the corresponding optimal trajectories.

In this paper, we focus on the convergence of feedback Nash equilibria in the finite state space scenario.
We follow  the approach of \cite{card}, showing the convergence of the value functions of the $N$-player game to the solution to the Master Equation.  The argument provides also the convergence of the feedback Nash equilibria and a propagation of chaos property for the associated optimal trajectories. The coupling technique necessary for the proof is the main motivation for writing the dynamics of the $N$ players as stochastic differential equations driven by Poisson Random measures, as in \cite{alekos}. 




In order to motivate the present work, let us  introduce the equations in play at a formal level. The dynamics of the $N$-player game are given by the system of controlled SDEs:
\begin{equation}
\label{eqn:sde}
X_{i}(t)  = Z_i + \int_0^t \int_\Xi f( X_i(s^{-}), \xi, \alpha^i(s,\bm{X}_{s^-}))\mathcal{N}_i(ds,d\xi),
\end{equation}
for $i = 1,\dots,N$, where each $X_i(t)$ is a process taking values in the finite space $\Sigma = \left\{1,\dots,d\right\}$ and we denote
by  $\bm{X}_t := (X_1(t), \dots, X_N(t))$ the vector of the $N$ processes; $\mathcal{N}_i$ are $N$ i.i.d. Poisson measures on 
$[0,T] \times \Xi$, with $\Xi\subset \mathbb{R}^d$, and the controls $\alpha^i \in A\subset \mathbb{R}^d$ are only in feedback form.
The function $f$ is crucial for the definition of the dynamics \eqref{eqn:sde}: it models the possible jumps of the Markov chain, while the Poisson measures prescribe their random occurrences. 
Following an idea of \cite{Graham}, we define the function $f$ so that 
the control
$\alpha^i_y(t,x, \bm{x}_t^{N,i})$ represents the rate at which player $i$ decides to go from state $x$ to state $y$, when $x\neq y$, 
$\bm{x}_t^{N,i}$ being the states of the other $N-1$ players at time $t$; c.f. \eqref{f} and \eqref{tran} below.
Let us remark that, while Cardaliaguet et al. study the convergence problem also in the presence of a noise (Brownian motion) common to all the players, which makes things even more difficult, we do not consider here any common noise. In the discrete setting, this would result in considering dynamics with simultaneous jumps, which can be realized by adding another Poisson measure in \eqref{eqn:sde}, common to all the players.

In our framework, we show that there exists a unique feedback Nash equilibrium for the $N$-player game. It is provided by the Hamilton-Jacobi-Bellman system of $N d^N$ coupled ODE's
\begin{equation}
\tag{HJB}
\begin{cases}
-\frac{\partial {v}}{\partial t}^{N,i}(t,\bm{x}) - \sum_{j=1, \ j\neq i}^N \alpha^*( x_j, \Delta^j v^{N,j}) \cdot \Delta^j v^{N,i} + H(x_i, \Delta^i v^{N,i}) = F^{N,i}(\bm{x}),\\
v^{N,i}(T, \bm{x}) = G^{N,i}(\bm{x}).
\label{eqn: HJB}
\end{cases}
\end{equation}
In the above equation, $F^{N,i}$ and $G^{N,i}$ are respectively the running and terminal costs, 
$H$ is the Hamiltonian and $\alpha^*$ its unique maximizer, and 
\begin{equation*}
\Delta^j g(\bm{x}):=  \left( g(x_1, \dots, y, \dots, x_N) - g(x_1, \dots, x_j, \dots, x_N)\right)_{y = 1,\dots,d} \in \mathbb{R}^{d}
\end{equation*}
 denotes the finite difference of a function $g(\bm{x}) = g(x_1, \dots, x_N)$ with respect to its $j$-th entry.

The study of convergence consists in finding a limit for the (\ref{eqn: HJB}) system as $N$ tends to infinity.
To this end, we assume symmetric properties of the game. Namely, the costs $F^{N,i}$ and $G^{N,i}$ satisfy the mean field assumptions, i.e. there exist two functions $F$ and $G$ such that $F^{N,i}(\bm{x}) = F(x_i, m_{\bm{x}}^{N,i})$ and $G^{N,i}(\bm{x}) = G(x_i, m_{\bm{x}}^{N,i})$, where $m_{\bm{x}}^{N,i}$ denotes the empirical measure of all the players except for the $i$-th, which belongs to $P(\Sigma)$, the space of probability measures on $\Sigma$.
Thanks to these mean field assumptions, we shall say 
that the solution $v^{N,i}$ of the (\ref{eqn: HJB}) system can be found in the form $v^{N,i}(t, \bm{x}) = V^N(t, x_i, m_{\bm{x}}^{N,i})$, for a suitable function $V^N$ of time, space and measure; this makes the convergence problem more tractable.
At a formal level, we can introduce the limiting equation assuming the existence of a function $U$ such that 
$V^N(t,x_i,m_{\bm{x}}^{N,i}) \sim U(t,x_i,m_{\bm{x}}^{N,i})$ for large $N$.  Then, let us analyze the different components of the HJB system and which should be their corresponding limits in terms of $U$.
First, the $i$-th difference of $v^{N,i}$ should converge to
\begin{align*}
\Delta^i v^{N,i}(t, \bm{x}) & = \left(v^{N,i}\left(t,y,m_{\bm{x}}^{N,i}\right) - v^{N,i}\left(t,x_i,m_{\bm{x}}^{N,i}\right)\right)_{y=1,\dots,d} \\
& \to \left(U(t,y,m) - U(t,x_i, m)\right)_{y=1,\dots,d} = \Delta^x U(t,x_i,m).
\end{align*}
For $j\neq i$ we should instead get
\begin{align*}
\Delta^j v^{N,i}(t,\bm{x}) &= \left(v^{N,i}\left(t,x_i, \frac{1}{N-1} \sum_{k\neq j,i} \delta_{x_k} + \frac{1}{N-1} \delta_y\right) - v^{N,i}\left(t,x_i, \frac{1}{N-1} \sum_{k \neq i} \delta_{x_k}\right)\right)_{y = 1,\dots, d} \\ 
& \sim \frac{1}{N-1} D^m U(t,x_i, m_{\bm{x}}^{N,i}, x_j),
\end{align*}
modulo terms of order $O(1/N^2)$, where a precise definition of $D^m U$, the derivative with respect  to a probability measure, will be given in the next section.
Then,  $H(x_i, \Delta^i v^{N,i}) \to H(x_i, \Delta^x U)$, and we should obtain
\begin{align*}
\sum_{j=1, j \neq i}^N & \alpha^{*}(x_j, \Delta^j v^{N,j})  \cdot \Delta^j v^{N,i} \sim \frac{1}{N-1} \sum_{j=1, j \neq i}^N \alpha^*(x_j, \Delta^x U(t,x_j, m_{\bm{x}}^{N,i})) \cdot D_m U(t,x_i, m_{\bm{x}}^{N,i}, x_j) \\
& \sim \int_{\Sigma} \alpha^*(y, \Delta^y U(t,y,m_{\bm{x}}^{N,i})) \cdot D^m U(t,x_i, m_{\bm{x}}^{N,i}, y) d m_{\bm{x}}^{N,i}(y) \\
& \to \int_{\Sigma} D^m U (t,x,m,y)\cdot \alpha^*(y, \Delta^y U(t,y,m)) dm(y).
\end{align*}
Thus, we are able to introduce the Master Equation, that is the equation to which we would like to prove convergence 
\begin{equation}
\tag{M}
\begin{cases}
& -\frac{\partial U}{\partial t}  + H(x, \Delta^x U) - \int_{\Sigma} D^m U(t,x,m,y) \cdot \alpha^*(y, \Delta^y U(t,y,m)) dm(y) = F(x,m),  \\
& U(T,x,m) = G(x,m), \ \ \ (x,m) \in \Sigma \times P(\Sigma), \ \ \ t\in[0,T].
\end{cases}
\label{eqn:M}
\end{equation}

It is a first order PDE in $P(\Sigma)$, the simplex of probability measures in $\mathbb{R}^d$.
 We solve it using the strategy developed in \cite{card}, which relies on the method of characteristics.  Indeed, as remarked above, the classical Mean Field Game system can be seen as the characteristic curves of \eqref{eqn:M}. In our finite space setting, the Mean Field Game system consists of two coupled ODEs: a Hamilton-Jacobi-Bellman equation giving the value function of the limiting control problem and a Kolmogorov-Fokker-Planck describing the evolution of the limiting deterministic flow of probability measures. We solve the Mean Field Game system for any initial time and initial distribution: this defines a candidate solution to \eqref{eqn:M} and, in order to prove that it is 
differentiable with respect to the initial condition, we introduce and analyze  a linearized Mean Field Game system. 
To prove the well posedness of \eqref{eqn:M} for any time horizon, the sufficient hypotheses we make are the monotonicity assumptions of Lasry and Lions. However, we stress again that these assumptions play no role in the convergence argument, as it requires only the existence of a regular solution to \eqref{eqn:M}.

Let us mention that finite space Mean Field Games have been studied by several authors in the last years. 
An equation similar to \eqref{eqn:M}, but holding in the whole space $\mathbb{R}^d$, was studied in \cite{lezioni}, proving the well-posedness and regularity under stronger assumptions. The well posedness of the Mean Field Game system was discussed in  \cite{gomes_mohr} - in the discrete time framework -, in \cite{gomes} and in \cite{gueant}, the latter with applications on graphs. 
The works \cite{ferreira} and \cite{gomes} deal also with the problem of convergence, as $T$ tends to infinity, to the stationary  Mean Field Game. The Master Equation was discussed - but only at a formal level - in \cite{gomes_velho}, \cite{gomes_velho2} and again in 
\cite{gomes}, in the first two with a particular focus on the two state problem. A new class of Mean Field Games with major and minor agents was analyzed in \cite{carW}, showing the relation with the $N$-player game in the approximation direction. 

Here, we also study the empirical measure process of the $N$-player optimal trajectories. Indeed, the convergence obtained allows to get a Central Limit Theorem and a Large Deviation Principle  for the asymptotic behavior, as $N$ tends to infinity, of such process. 
These results are, to the best of our knowledge, new in Mean Field Game theory and allow to better understand the convergence of the empirical distribution of the players to the limit measure. Both Central Limit Theorems and Large Deviation Principles have been investigated for classical finite state mean field interacting particle systems, i.e. where the prelimit jump rates are the same for any individual. The fluctuations are derived e.g. in  \cite{comets}, via the martingale problem, and in \cite{daipra}, by  analyzing the infinitesimal generator. 
The large deviation properties are found e.g. in \cite{25}, \cite{32}, \cite{31}, and, more recently, in \cite{DRW}.

The key point for proving these results is to compare the prelimit optimal trajectories with the ones in which each player chooses the control induced by the Master Equation. The fluctuations are then found by analyzing the associated infinitesimal generator, while the Large Deviation properties are derived using a result in \cite{DRW}. It is worth saying that the limit processes involve the solution to the Master Equation. 
Let us mention that such properties are being studied in the diffusion case, independently, via the Master Equation approach, by Lacker et al. \cite{slides}.

Finally, we mention a very recent preprint \cite{eb}, which appeared some days after submission of this paper. 
In this work, independently, the authors again use the Master Equation approach to find the same convergence result we prove here. They still rely on the idea developed in \cite{card}, but  consider a probabilistic representation of the dynamics different from ours. Moreover, they also obtain a Central Limit Theorem for the fluctuations of the empirical measure processes. However, they prove it in a different way, that is, via a martingale Central Limit Theorem.

\subsection*{Structure of the paper}

The rest of the paper is organized as follows. 
In Section \ref{model}, we start with the notation and the definition of derivatives in the simplex. So we present
the two sets of assumptions we make use of: one for the convergence, the fluctuations and the large deviation results, while the other, stronger, for the well posedness of the Master Equation; we also show an example in which the assumptions are satisfied. Then we give a
detailed description of both the $N$-player game and the limit model. 
Section \ref{convergence} contains the convergence results and their proofs, while in Section \ref{fl} we employ the convergence argument to derive the asymptotic behaviour of the empirical measure process, that is, the Large Deviation Principle and the Central Limit Theorem. 
 Section \ref{mastersection} analyzes the well-posedness and regularity of the solution to the Master Equation. We conclude with Section 
\ref{concl} by summarizing all the main results.

\section{Model and assumptions}
\label{model}

\subsection{Notation}

Here we briefly clarify the notations used throughout the paper. First of all, we are considering $\Sigma = \left\{1,\dots, d\right\}$ to be the finite state space of any player.  Let $T$ be the finite time horizon and $A :=[\kappa,M]^d$, for $\kappa, M > 0$, be the compact space of control values.
Denote by
\begin{equation*}
P(\Sigma) := \left\{m \in \mathbb{R}^d \ : \ m_j \geq 0, \ \ m_1 + \dots + m_d = 1\right\}
\end{equation*}
the space of probability measures on $\Sigma$. 
Besides the euclidean distance in $\mathbb{R}^d$, denoted with $|\cdot|$, we may interchangeably use the Wasserstein metric $\bm{d}_1$ on $P(\Sigma)$  since all metrics are equivalent. We observe that the simplex $P(\Sigma)$ is a compact and convex subset of $\mathbb{R}^d$.

Let $\Xi := [0,M]^d$. In the dynamics given by \eqref{eqn:sde},  the function 
$f : \Sigma \times \Xi \times A  \to \left\{-d,\dots, d\right\}$ modeling the jumps has to be a measurable function such that $f(x,\xi,a) \in \left\{1-x, \dots, d-x\right\}$. Specifically, throughout the paper we set, for $x \in \Sigma, \ \xi = (\xi_y)_{y \in \Sigma}$ and $a = (a_y)_{y \in \Sigma}$, 
\begin{equation}
\label{f}
f(x,\xi,a) := \sum_{y \in \Sigma} (y-x) \mathbbm{1}_{]0, a_y[} (\xi_y).
\end{equation}
The measures $\mathcal{N}_i$ appearing in \eqref{eqn:sde} are $N$ i.i.d. stationary Poisson random measures on $[0,T] \times \Xi$, with intensity measure $\nu$ on $\Xi$ given by 
\begin{equation}
\label{nu}
\nu(E) := \sum_{j=1}^d \ell(E \cap \Xi_j),
\end{equation}
for any $E$ in the Borel $\sigma$-algebra $\mathcal{B}(\Xi)$ of $\Xi$, where $\Xi_j := \left\{ u \in \Xi \ : \ u_i = 0 \ \ \forall \ i \neq j\right\}$ is viewed as a subset of $\mathbb{R}$, and $\ell$ is the Lebesgue measure on $\mathbb{R}$.  
We fix a probability space $(\Omega, \mathcal{F}, \mathbb{P})$ and denote by 
$\mathbb{F}=(\mathcal{F}_t)_{t\in [0,T]} $ the filtration generated by the Poisson measures. These definitions of $f$ and $\nu$ ensure that the control is exactly the transition rate of the Markov chain; see \eqref{tran} below. 

The initial datum of the $N$-player game is represented by $N$ i.i.d. random variables $Z_1, \dots, Z_N$ with values in $\Sigma$
 and distributed as $m_0 \in P(\Sigma)$. The vector $\bm{Z}= (Z_1,\ldots, Z_N)$ is in particular \emph{exchangeable}, in the sense that the joint distribution is invariant under permutations, and is assumed to be $\mathcal{F}_0$-measurable, i.e. independent of the noise.

The state of player $i$ at time $t$ is denoted by $X_{i}(t)$, with $\bm{X}_t :=(X_1(t),\dots,X_N(t))$. The trajectories of each $X_i$ are  in $D([0,T], \Sigma)$,  the space of càdlàg functions from $[0,T]$ to $\Sigma$ endowed with the Skorokhod metric.
For $\bm{x}=(x_1,\dots,x_N)\in\Sigma^N$, denote the empirical measures
$$m_{\bm{x}}^N := \frac{1}{N}\sum_{j=1}^N \delta_{x_j} \qquad m^{N,i}_{\bm{x}} :=  \frac{1}{N-1}\sum_{j=1, j \neq i}^N \delta_{x_j}.$$
Thus, $m^N_{\bm{X}}(t):= m^N_{\bm{X}_t}$ is the empirical measure of the  $N$ players and $m^{N,i}_{\bm{X}}(t):= m^{N,i}_{\bm{X}_t} $ is the empirical measure of all the players except the $i$-th. Clearly, they are $P(\Sigma)$-valued stochastic processes. In the limiting dynamics, the empirical measure is replaced by a deterministic flow of probability measures $m :[0,T] \to P(\Sigma)$.

In choosing his/her strategy, each player minimizes the sum of three costs: a Lagrangian $L:\Sigma \times A \longrightarrow \mathbb{R}$, a running cost $F: \Sigma\times P(\Sigma)\longrightarrow \mathbb{R}$ and a final cost $G: \Sigma\times P(\Sigma)\longrightarrow \mathbb{R}$ (see next section for the precise definition of the $N$-player game).
The Hamiltonian $H$ is defined as the Legendre transform of $L$:
\begin{equation}
\label{h}
H(x, p) := \sup_{\alpha \in A} \left\{- \alpha \cdot p - L(x,\alpha)\right\},
\end{equation}
for $x \in \Sigma$ and $p \in \mathbb{R}^d$.


Given a function $ g: \Sigma \to \mathbb{R}$ we denote its first finite difference $\Delta g (x) \in \mathbb{R}^{d}$  by
\begin{equation*}
\Delta  g(x):= \begin{pmatrix}
g(1) - g(x) \\
\cdot \\
\cdot \\
\cdot \\
g(d) - g(x)
\end{pmatrix}.
\end{equation*}
When we have a function $g : \Sigma^N \to \mathbb{R}$, we denote with $\Delta^j g(\bm{x}) \in \mathbb{R}^{d}$ the first finite difference with respect to the $j$-th coordinate, namely
\begin{equation*}
\Delta^j  g(\bm{x}):= \begin{pmatrix}
g(x_1,\dots, x_{j-1},1, x_{j+1}, \dots, x_N) - g(\bm{x}) \\
\cdot \\
\cdot \\
\cdot \\
g(x_1,\dots, x_{j-1},d, x_{j+1}, \dots, x_N) - g(\bm{x})
\end{pmatrix}.
\end{equation*}
 For future use, let us observe that, for $g : \Sigma \to \mathbb{R}$, 
\be
|\Delta g(x)| \leq \max_y [\Delta g(x)]_y \leq 2 \max_x |g(x)| \leq C|g|.
\label{bounddelta}
\ee
For a function $u:[t_0,T]\times\Sigma \longrightarrow \mathbb{R}$, we denote 
\be
||u|| := \sup_{t \in [t_0,T]} \max_{x \in \Sigma} |u(t,x)|.
\label{norm}
\ee 
We also use the notation $u(t) :=  (u_1(t), \dots, u_d(t)) = (u(t,1), \dots, u(t,d))$. When considering a function $u$ with values in 
$\mathbb{R}^d$, its norm is defined as in \eqref{norm}, but where $|\cdot|$ denotes the euclidean norm in $\mathbb{R}^d$.

We now introduce the concept of variation with respect to a probability measure $m$ of a function $U:P(\Sigma) \to \mathbb{R}$. 
Let us remark that the usual notion of gradient cannot be defined for such a function: since the domain is  $P(\Sigma)$ we are not allowed to define e.g. the directional derivative $\frac{\partial}{\partial m_1}$, as we would have to extend the definition of $U$ outside the simplex.  
\begin{defn}
We say that a function $U: P(\Sigma) \to \mathbb{R}$ is \textit{differentiable} if there exists a function $D^m U : P(\Sigma) \times \Sigma \to \mathbb{R}^{d}$ given by
\begin{equation}
[D^m U(m,y)]_z := \lim_{s \to 0^+} \frac{U(m + s(\delta_z - \delta_y)) - U(m)}{s}.
\end{equation}
for $z = 1, \dots, d$. Moreover, we say that $U$ is $C^1$ if the function $D^m U$ is continuous in $m$.
\end{defn}

Morally, we can think of $[D^m U(m,y)]_z$ as the (right) directional derivative of $U$ with respect to $m$ 
along the direction $\delta_z - \delta_y$.
We also observe that $m + s(\delta_z - \delta_y)$ might be outside the probability simplex (e.g. when we are at the boundary), in which case we consider the limit only across admissible directions. However, note that, for our purposes, this is not really a problem: since in the limit $m(t)$ will be the distribution of the reference player, the bound from below for the control ensures that the boundary of the simplex will never be touched. 

Together with the definition, we state an identity which will come useful in the following sections:
\begin{equation}
\label{eqn:identity}
[D^m U(m,y)]_z = [D^m U(m,x)]_z + [D^m U(m,y)]_x,
\end{equation} 
for any $x,y, z\in \Sigma$.
Its derivation is an immediate consequence of the linearity of the directional derivative.

We can easily extend the above definition to the case of derivative with respect to a direction $\mu \in P_0(\Sigma)$, with 
$$P_0(\Sigma) := \left\{\mu \in \mathbb{R}^d \ : \ \mu_1 + \dots + \mu_d = 0\right\}.$$
Indeed, an element $\mu= (\mu_1, \dots, \mu_d) = \sum_{z\in\Sigma} \mu_z\in P_0(\Sigma)$ can be rewritten as a linear combination of $\delta_z - \delta_y$ as follows
\begin{equation*}
\mu = \sum_{z\neq y} \mu_z (\delta_z - \delta_y),
\end{equation*}
for  each $y\in \Sigma$, since $\sum_{z\neq y} \mu_z (\delta_z - \delta_y) =
 \sum_{z \neq y} \mu_z \delta_z - \left(\sum_{z \neq y} \mu_z\right) \delta_y$, and $\sum_{z \neq y} \mu_z = - \mu_y$.

This remark allows us to define the derivative of $U(m)$ along the direction $\mu \in P_0(\Sigma)$ as a map $\frac{\partial}{\partial \mu} U : P(\Sigma) \times \Sigma \to \mathbb{R}$, defined for each $y\in\Sigma$ by 
\be
\frac{\partial}{\partial \mu} U(m, y) := \sum_{z \neq y}^d \mu_z \left[D^m U(m,y)\right]_z = \mu \cdot D^m U(m,y),
\ee
 where the last equality comes from the fact that $[D^m U(m,y)]_y = 0$.

We also note that the definition of $\frac{\partial}{\partial \mu} U(m, y)$ does not actually depend on $y$, i.e.
\begin{equation}
\label{eqn:identity2}
\frac{\partial}{\partial \mu} U(m, y) = \frac{\partial}{\partial \mu} U(m, 1)
\end{equation}
for every $y \in \Sigma$ and for this reason we will fix $y=1$ when needed in the equations. Indeed,  by means of identity \eqref{eqn:identity}   and the fact that $\mu \in P_0(\Sigma)$, for each $y \in \Sigma$
\begin{align*}
\frac{\partial}{\partial \mu} U(m, 1) & = \sum_{z = 1}^d \mu_z \left[D^m U(m,1)\right]_z
  = [\text{identity \eqref{eqn:identity}}] = \sum_{z=1}^d\left([D^m U(m,y)]_z + [D^m U(m,1)]_y\right) \mu_z \\
& = \sum_{z=1}^d [D^m U(m,y)]_z \mu_z + [D^m U(m,1)]_y \sum_{z=1}^d \mu_z
 = \sum_{z=1}^d [D^m U(m,y)]_z \mu_z = \frac{\partial}{\partial \mu} U(m, y).
\end{align*}

For a function $U : \Sigma \times P(\Sigma) \to \mathbb{R}$ we denote the variation with respect to the first coordinate in a point $(x,m) \in \Sigma \times P(\Sigma)$ by $\Delta^x U(x,m)$. Also, denote by $\Gamma^\dagger$ the transpose of a matrix $\Gamma$.

\subsection{Assumptions}

We now summarize the assumptions we make, which can vary according to the different results. 

Because of the compactness of $A$, the continuity of $L$ with respect to its second argument is sufficient for guaranteeing the existence and finiteness of the supremum in \eqref{h} for each $(x,p)$. Moreover, we assume that there exists a unique maximizer $\alpha^*(x,p)$ in the definition of $H$ for every $(x,p)$:
\be
 \alpha^*(x, p) := \arg \min_{\alpha \in A}\left\{L(x,\alpha) + \alpha \cdot p\right\} = \arg \max_{\alpha \in A}\left\{-L(x,\alpha) -\alpha \cdot p\right\}.
\label{stella}
\ee
With our choices for $f$ in \eqref{f} and the intensity measure $\nu$ in \eqref{nu}, a sufficient condition for the above assertion is given by the strict convexity of $L$ in $\alpha$ (see Lemma 3 in \cite{alekos}).
If $L$ is uniformly convex, such optimum $\alpha^*$ is globally Lipschitz in $p$, and whenever $H$ is differentiable it can be explicitly expressed as $\alpha^*(x, p) = -D_p H(x,p)$; see Proposition 1 in \cite{gomes} for the proof.

We will work with two sets of assumptions on $H$. 
We first observe that it is enough to give hypotheses for $H(x,\cdot)$ on a sufficiently big compact subset of $\mathbb{R}^d$, i.e. for $|p|\leq K$, because of the uniform boundedness of $\Delta^i v^{N,i}$: see next section for details (Remark \ref{boundv}). In what follows, the constant $K$ is fixed:
\begin{itemize}
\item[\textbf{(H1)}]
If $|p|\leq K$ then $H$ and $\alpha^*$ are Lipschitz continuous in $p$.
\end{itemize}
We stress the fact that the above assumptions, together with the existence of a regular solution to \eqref{eqn:M}, are alone sufficient for proving the convergence of the $N$-player game to the limiting mean-field game dynamics.

In order to establish the well-posedness and the needed regularity for the Master Equation we make use of the following additional assumptions:
\begin{itemize}
\item[\textbf{(RegH)}] If $|p|\leq K$, $H$ is $C^2$ with respect to $p$; $H$, $D_p H$ and $D^2_{pp}H$ are Lipschitz in $p$ and the second derivative is bounded away from $0$, i.e.  there exists a constant $C$ such that
\begin{equation}
D_{pp}^2 H(x,p) \geq C^{-1};
\label{bound}
\end{equation}

 \item[\textbf{(Mon)}] The cost functions $F$ and $G$ are monotone in $m$ in the Lasry-Lions sense, i.e., for every $m, m' \in P(\Sigma)$, 
 \begin{equation}
 \label{mon}
 \sum_{x \in \Sigma} (F(x,m) - F(x,m'))(m(x) - m'(x)) \geq 0, 
 \end{equation}
and the same holds for $G$;
\item[\textbf{(RegFG)}] The cost functions $F$ and $G$ are $C^1$ with respect to $m$, with $D^m F$ and $D^m G$ bounded and Lipschitz continuous. In this case \eqref{mon} is equivalent to say that
 \begin{equation}
 \sum_x \mu_x [D^m F(x,m, 1) \cdot \mu] \geq 0
\end{equation}
for any $m\in P(\Sigma)$ and $\mu\in P_0(\Sigma)$.
\end{itemize}
Observe that the assumptions on $H$ allow for quadratic Hamiltonian. As we will see, the above assumptions imply both the boundedness and Lipschitz continuity of $\Delta^x U$ and $D^m U$ with respect to $m$. We conclude the section with an example for which all the assumptions are satisfied.

\begin{exmp}
\normalfont
The easiest example for the costs $F$ and $G$ is $F(x,m) = G(x,m) = m(x)$. Slightly more in general, one can consider 
$F(x,m) = \nabla \phi(m)(x)$, $\phi$ being a real convex function on $\mathbb{R}^d$.

For the choice of the Lagrangian $L$, a bit of work is needed in order to recover the regularity for $H$, since the maximization in the definition \eqref{h} of $H$ is performed only on the compact subset $A = [\kappa,M]^d$ of $\mathbb{R}^d$.

Consider the Lagrangian, not depending on $x$, defined by 
\begin{equation}
L(\alpha):= b|\alpha - a|^2,
\end{equation}
with $a := \left(\frac{\kappa + M}{2}\right)(1,\dots,1)^\dagger$ and $b$ a large enough constant to be chosen later.
The computation of $H := \sup_{\alpha \in [\kappa, M]} \left\{-p \cdot \alpha - L(\alpha)\right\}$ for such choice of $L$ gives
\begin{equation}
\label{eqn:h}
H(p) = \frac{{p}^2}{4b} - a \cdot p,
\end{equation}
for $|p| \leq b(M-\kappa)$, while  $H$ is linear outside this interval. It is trivial to verify that  $H$ is in $C^1(\mathbb{R}^d)$,  and thus \textbf{(H1)} is satisfied, while $H$ is not in $C^2(\mathbb{R}^d)$  because of the linear components. 
Nevertheless, (\ref{bound}) is satisfied whenever $|p| \leq K$, with the choice $b := \frac{K}{M-\kappa}$. Moreover, the Lipschitz continuity of $D_p H$ and $D_{pp}^2 H$ is trivially holding because of expression \eqref{eqn:h} for $|p| \leq K$ and the linearity outside, and \textbf{(RegH)} follows.
Note that $p$ represents the gradient of the value functions and thus it belongs to a compact $[-K,K]$, where $K$ is independent of $b$; c.f. Remark 1 below.

\end{exmp}

\subsection{N-Player Game} 

In this section we describe the $N$-player game in a general setting. Namely, we suppose that each individual has complete information on the states of all the other players and we do not require the players to be symmetric. Then, we show the relation between system \eqref{eqn: HJB} and the concept of Nash equilibria for the game through a classical Verification Theorem. We conclude the section by introducing the mean field assumptions and stating a consequence on the symmetry of the solution to \eqref{eqn: HJB}.
We remark that most of the results of this section were found also in \cite{gomes}, but in a slightly different framework. Namely, there the authors assumed a priori that the value functions depend on the empirical measure, assuming hence symmetry. Moreover, they studied the infinitesimal generator of the processes, while here we employ our probabilistic representation. 

In the prelimit the dynamics are given by the system of $N$ controlled SDEs 
\be
X_{i}(t)  = Z_i + \int_0^t \int_\Xi f(X_i(s^{-}), \xi, \alpha^i(s,\bm{X}_{s^-}))\mathcal{N}_i(ds,d\xi),
\label{Nsde}
\ee
for $i = 1,\dots, N$, where $f$ is given by \eqref{f} and $\bm{X}_t= (X_1(t),\dots,X_N(t))$.
Each player is allowed to choose his/her control $\alpha^i$ having complete information on the state of the other players. 
We consider only controls
$\boldsymbol{\alpha}^N := (\alpha^1, \dots, \alpha^N)$ in feedback form, i.e. the controls are deterministic functions 
of time and space $\alpha^i :[0,T] \times \Sigma^N \longrightarrow A$, $\alpha^i = \alpha^i(t, \bm{x})$. We say that 
$\alpha^i\in\mathcal{A}$, for each $i$, if it is a measurable function of time. We denote by $\mathcal{A}^N$ the set of feedback strategy vectors $\boldsymbol{\alpha}^N = (\alpha^1, \dots, \alpha^N)$, each $\alpha^i$ belonging to $\mathcal{A}$. 


We remark that the dynamics  \eqref{Nsde} is always well-posed, for any admissible choice of the control, since the state space is finite and the coefficients are then trivially Lipschitz continuous. Namely, for any $\boldsymbol{\alpha}^N \in \mathcal{A}^N$ there exists a unique strong solution to \eqref{Nsde}, in the sense that $(\bm{X}_t)_{t\in [0,T]}$ is adapted to the filtration $\mathbb{F}$ generated by the Poisson random measures. 


With the definition of $f$ in \eqref{f} and the intensity measure $\nu$ in \eqref{nu}, the dynamics of any player remains in $\Sigma$ for any time and the feedback controls are exactly the transition rates of the continuous time Markov chains $X_i(t)$. Indeed, one can prove - see \cite{alekos} - that, for $x \neq y$ and $\bm{x}^{N,i} \in \Sigma^{N-1}$,
\be
\mathbb{P}\left[X_i(t+h) = y \big| X_i(t) = x, \ \bm{X}_t^{N,i} = \bm{x}^{N,i}\right] = \alpha^i_y(t, x, \bm{x}^{N,i}) h + o(h).
\label{tran}
\ee
Since $\alpha$ is the vector of the transition rates of the Markov chain,  we set $\alpha^i_x(x) = -\sum_{y \neq x} \alpha^i_y(x)$.
We remark that the boundedness from below of the controls ($\alpha^i \in [\kappa,M]^d$, $\kappa >0$) guarantees 
 that $P(X_i(t) = x) > 0$ for every $x$ in $\Sigma$ and $t>0$, for any player $i$.

Next, we define the object of the minimization.
Let $\boldsymbol{\alpha}^N = (\alpha^1, \dots, \alpha^N) \in \mathcal{A}^N$ be a strategy vector and $\bm{X} = (X_1, \dots, X_N)$ the corresponding solution to \eqref{Nsde}. For $i = 1,\dots, N$ and given functions $F^{N,i}, G^{N,i}: \Sigma^N \longrightarrow \mathbb{R}$, we associate to the $i$-th player the cost functional
\begin{equation}
J_i^N(\boldsymbol{\alpha}^N) := \mathbb{E}\left[\int_{0}^T \left[L(X_i(t), \alpha^i(t,\bm{X}_t )) + F^{N,i}(\bm{X}_t)\right]dt + G^{N,i}(\bm{X}_T)\right].
\end{equation}
The optimality condition for the $N$-player game is given by the usual concept of Nash equilibria. 
For a strategy vector $\boldsymbol{\alpha}^N = (\alpha^1, \dots, \alpha^N)\in\mathcal{A}^N$ and   $\beta\in\mathcal{A}$, denote by $[\boldsymbol{\alpha}^{N,-i}; \beta]$ the perturbed strategy vector given by
\begin{equation*}
[\boldsymbol{\alpha}^{N,-i}; \beta]^j := 
\begin{cases}
\alpha^j, \ \ j \neq i\\
\beta, \ \ j = i.
\end{cases}
\end{equation*}
Then, we can introduce the following
\begin{defn}
A strategy vector $\boldsymbol{\alpha}^N$ is said to be a Nash equilibrium for the $N$-player game if for each $i=1, \dots, N$
\begin{equation*}
J_i^N(\boldsymbol{\alpha}^N) = \inf_{\beta\in\mathcal{A}}J_i^N([\boldsymbol{\alpha}^{N,-i}; \beta]).
\end{equation*}
\end{defn}

Let us now introduce the functional
\be
J_i^N (t,\bm{x},\boldsymbol{\alpha}^N) := \mathbb{E}\left[\int_{t}^T [L(X^{t,\bm{x}}_i(s), \alpha^i(s,\bm{X}^{t,\bm{x}}_s )) + 
F^{N,i}(\bm{X}^{t,\bm{x}}_s)]ds + G^{N,i}(\bm{X}^{t,\bm{x}}_T)\right],
\ee
where
$$
X^{t,\bm{x}}_{i}(s)  = x_i + \int_t^s \int_\Xi f(X^{t,\bm{x}}_i(r^{-}), \xi, \alpha^i(r,\bm{X}^{t,\bm{x}}_{r^-}))\mathcal{N}_i(dr,d\xi)
\qquad s \in [t,T].
$$
We work under hypotheses that guarantee the existence of a unique maximizer $\alpha^*(x, p)$ defined in \eqref{stella}.
With this notation, the Hamilton-Jacobi-Bellman system associated to the above differential game is given by the system \eqref{eqn: HJB} presented in the Introduction:
\begin{equation*}
\begin{cases}
-\frac{\partial {v}}{\partial t}^{N,i}(t,\bm{x}) - \sum_{j=1, \ j\neq i}^N \alpha^*(x_j, \Delta^j v^{N,j}) \cdot \Delta^j v^{N,i} + H(x_i, \Delta^i v^{N,i}) = F^{N,i}(\bm{x}),\\
v^{N,i}(T, \bm{x}) = G^{N,i}(\bm{x}).
\end{cases}
\end{equation*}
This is a system of $N d^N$ coupled ODE's, whose well-posedness  for all $T > 0$ can be proved through standard ODEs techniques, because of the Lipschitz continuity of the vector fields involved in the equations.

We are now able to relate system \eqref{eqn: HJB} to the Nash equilibria for the $N$-player game through the following 
\begin{prop}[Verification Theorem]
Let $v^{N,i}$, $i = 1,\dots, N$ be a classical solution to system \eqref{eqn: HJB}. Then the feedback strategy vector 
$\boldsymbol{\alpha}^{N*} = (\alpha^{1,*}, \dots, \alpha^{N,*})$ defined by
\be
\alpha^{i,*}(t,\bm{x}) := \alpha^{*}(x_i, \Delta^i v^{N,i}(t, \bm{x})) \qquad i=1,\ldots,N,
\ee
is the unique Nash equilibrium for the $N$-player game and the $v^{N,i}$'s are the \emph{value functions} of the game, i.e.
\be
v^{N,i}(t, \bm{x}) = J_i^N (t,\bm{x},\boldsymbol{\alpha}^{N*}) =  
\inf_{\beta\in\mathcal{A}}J_i^N(t,\bm{x},[\boldsymbol{\alpha}^{N*,-i}; \beta]).
\ee
\end{prop}
\begin{proof}
Let $\beta \in \mathcal{A}$ be any feedback and $\bm{X}^{t,\bm{x}}$ the corresponding solution to \eqref{Nsde}, given the strategy vector
$[\boldsymbol{\alpha}^{N*-i}; \beta]$; denote for simplicity $\bm{X}=\bm{X}^{t,\bm{x}}$.
 Fixing $i \in \left\{1,\dots,N\right\}$, because of the uniqueness of the maximizer in  \eqref{stella}, we have 
$$\frac{\partial {v}}{\partial t}^{N,i} + \sum_{j\neq i} \sum_{y = 1}^{d} \alpha^*_y(t, x_j, \Delta^j v^{N,j}) [\Delta^j v^{N,i}(t,\bm{x})]_y + 
\beta(t,\bm{x})\cdot \Delta^i v^{N,i}(t,\bm{x}) + L(x_i, \beta(t,\bm{x})) +  F^{N,i}(\bm{x}) \geq 0$$
 for any $t,\bm{x}$.
Applying first It\^o formula (Theorem II.5.1 in
\cite{IK}, p. 66) and then Lemma 3 in \cite{alekos} and the above inequality, we obtain 
\begin{align*}
v^{N,i}(t, \bm{x}) &= \mathbb{E} \left[v^{N,i}(T, \bm{X}_T) - \int_t^T \frac{\partial {v}}{\partial t}^{N,i}(s,\bm{X}_s)ds\right]\\
   &- \sum_{j=1}^N \mathbb{E} \left[\int_t^T \int_\Xi \left[v^{N,i}\big(X_1(s),\dots, X_j(s) + f(X_j(s),\xi,
	[\boldsymbol{\alpha}^{N*,-i}; \beta] (s,\bm{X}_s)),\dots,X_N(s)\big)\right.\right.\\
	&\left.\left.\quad \qquad -v^{N,i}(\bm{X}_s)\right]\nu(d\xi)ds\right]\\
	&= \mathbb{E} \left[v^{N,i}(T, \bm{X}_T) \right.\\
	& \left.-\int_t^T \left(\frac{\partial {v}}{\partial t}^{N,i}(s,\bm{X}_s)
	+ \sum_{j\neq i} \alpha^{j,*}(s,\bm{X}_s) \cdot \Delta^j v^{N,i}(s,\bm{X}_s) + \beta(t,\bm{X}_s)\cdot \Delta^i v^{N,i}(t,\bm{X}_s)\right)ds \right]\\
	&\leq \mathbb{E} \left[G^{N,i}(T, \bm{X}_T) +\int_t^T
	\left(L(X_i(s), \beta(s,\bm{X}_s)) + F^{N,i}(\bm{X}_s)\right)ds\right]\\
	&=: J_i^N(t,\bm{x},[\boldsymbol{\alpha}^{N*,-i}; \beta]).
\end{align*}
Replacing $\beta$ by $\alpha^{i*}$ the inequalities become equalities.
\end{proof}

\begin{rem}
\label{boundv}
It is important to observe that the solution $v^{N,i}$ to \eqref{eqn: HJB} is uniformly bounded with respect to $N$. Namely, there exists a constant $K > 0$ such that $\sup_{\bm{x} \in \Sigma^N} |v^{N,i}(t, \bm{x})| \leq K$, where the constant $K$ is independent of $N$, $i$ and $t$. This and \eqref{bounddelta} immediately imply an analogous bound for $|\Delta^i v^{N,i}(t,\bm{x})|$: it is for this reason that the only local regularity 
(assumptions \textbf{(H1)} and \textbf{(RegH)}) for $H(x,p)$ with respect to $p$ is enough for getting the convergence and the well-posedness results.
\end{rem}

We are interested in studying the limit of the (\ref{eqn: HJB}) system as $N \to \infty$ under symmetric properties for the $N$-player game. Namely, we assume that the players are all identical and indistinguishable. In practice, this symmetry is expressed through the following mean-field assumptions on the costs: 
\begin{align*}
\label{eqn:MF}
F^{N,i}(\bm{x}) & = F(x_i, m_{\bm{x}}^{N,i}), \\
\tag{M-F}
G^{N,i}(\bm{x}) & = G(x_i, m_{\bm{x}}^{N,i}),
\end{align*}
for some $F$ and $G : \Sigma \times P(\Sigma) \to \mathbb{R}$.
An easy but crucial consequence of assumptions \eqref{eqn:MF} and the uniqueness of solution to system \eqref{eqn: HJB} is that the solution $v^{N,i}$ of such system enjoys symmetric properties:

\begin{prop}
\label{simmetry}
Under the mean-field assumptions \eqref{eqn:MF}, there exists $v^N: [0,T] \times\Sigma^N\rightarrow\mathbb{R}^d$ such that the solutions $v^{N,i}$ to system \eqref{eqn: HJB} satisfy, for $i = 1,\dots, N$,
\begin{equation}
v^{N,i}(t, \bm{x}) = v^N(t, x_i, (x_1, \dots, x_{i-1}, x_{i+1}, \dots, x_N)),
\end{equation}
 for any $(t,x) \in [0,T] \times \Sigma$, and the function 
 $$
 \Sigma^{N-1} \ni (y_1, \dots, y_{N-1}) \to v^N(t,x, (y_1, \dots, y_{N-1}))
 $$
 is invariant under permutations of $(y_1, \dots, y_{N-1})$.
\end{prop}

\begin{proof}
Let $\tilde{\bm{x}}$ be defined from $\bm{x}$ after exchanging $x_k$ with $x_j$, for $j \neq k \neq i$.
Because of \eqref{eqn:MF}, we have that $F^{N,i}(\bm{x}) = F^{N,i}(\tilde{\bm{x}})$ and $G^{N,i}(\bm{x}) = G^{N,i}(\tilde{\bm{x}})$ and thus, by the uniqueness of solution to \eqref{eqn: HJB} we conclude $v^{N,i}(t,\bm{x}) = v^{N,i}(t,\tilde{\bm{x}})$.
\end{proof}
The above proposition motivates the study of a possible convergence of system \eqref{eqn: HJB} to a limiting system, by analyzing directly the limit of the functions $v^N$.


\subsection{Mean Field Game and Master Equation}

The limit as $N \to \infty$ would expectedly be characterized by a continuum of players in which the representative agent evolves according to the dynamics 
\begin{equation}
\label{eqn:ref_lim}
X(t)  = Z + \int_{0}^t \int_\Xi f( X(s^{-}), \xi, \alpha(s,X(s^-)))\mathcal{N}(ds,d\xi), \ \ t \in [0,T],
\end{equation}
where the law of the initial condition $Z$ is $m_0$ and $\mathcal{N}$ is a Poisson Random measure with intensity measure $\nu$ defined in \eqref{nu}. The controls are feedback in $\mathbb{A}$, which denotes the space of measurable functions $\alpha:[0,T]\times\Sigma\longrightarrow A$.
The empirical measure of the $N$ players is replaced by a deterministic flow of probability measures 
$m : [0,T] \to P(\Sigma)$. The associated cost is 
\be
J(\alpha,m):= \mathbb{E} \left[ \int_0^T
	\left[L(X(t), \alpha(t,X(t))) + F(X(t),m(t))\right]dt +G(X(T), m(T))\right].
\ee 

In literature, such limiting dynamics are described by the celebrated Mean Field Game system, whose unknowns are two functions $(u,m)$. The equation in $u$ describes the dynamics of the value function of the reference player, which optimizes his/her payoff under the influence of the collective behavior of the others, while the equation in $m$ describes the evolution of the distribution of the players. 
In our discrete setting the Mean Field Game system takes the following form of a strongly coupled system of ODEs:
\begin{equation}
\tag{MFG}
\begin{cases}
-\frac{d}{dt} u(t,x) + H(x, \Delta^x u(t,x)) =  F(x, m(t)), \\
\frac{d}{dt} m_x(t) = \sum_y m_y(t) \alpha^{*}_x(y, \Delta^y u(t,y)),\\
u(T,x) = G(x, m(T)), \\
m_x(t_0) = m_{x,0},
\end{cases}
\label{eqn:MFG}
\end{equation}
with $\alpha^{*}(x,p)$  defined in \eqref{stella} and $u,m:[0,T]\times \Sigma \longrightarrow \mathbb{R}$. From a probabilistic point of view the solution $(u,m)$ can be seen as a fixed point: 
starting with a flow $m$ solve the first equation, the backward Hamilton-Jacobi-Bellman (HJB) equation for $u$, which yields a unique optimal feedback control $\alpha$ for $m$; then, impose that the flow of the corresponding solution  $X$ to \eqref{eqn:ref_lim} is exactly $m$, 
giving thus the second equation, the forward Kolmogorov-Fokker-Planck (KFP) for $m$. Hence, given a solution $(u,m)$ to \eqref{eqn:MFG},
 we have
$J(\alpha,m)\leq J(\beta,m)$ for any $\beta\in\mathbb{A}$, where $\alpha(t,x)= \alpha^{*}(x, \Delta^x u(t,x))$, and 
$Law(X(t))=m(t)$ for any $t\in [0,T]$. 

As already mentioned, recently in \cite{card}  a new technique involving the use of the so-called Master Equation was introduced to get the exact relation between symmetric $N$-Player Differential Games and Mean Field Games.
Generally speaking, the Master Equation summarizes all the information needed to get solutions to the Mean Field Game: namely, one can prove that the system \eqref{eqn:MFG} provides the characteristic curves for \eqref{eqn:M}. Indeed, $U(t_0, x, m_0) := u(t_0,x)$ solves \eqref{eqn:M}, $(u,m)$ being the solution to the Mean Field Game system \eqref{eqn:MFG} starting at time $t_0$ up to time $T$, with $m(t_0)=m_0$. 
Moreover, in the Introduction we already motivated heuristically the convergence result of system \eqref{eqn: HJB} to the Master Equation \eqref{eqn:M}. As it will be clear from the convergence argument, all that is needed is the existence of a regular solution to \eqref{eqn:M}.

To be specific on the needed regularity, we conclude this section with the definition of regular solution to \eqref{eqn:M}.
\begin{defn}
\label{def}
A function $U : [0,T] \times \Sigma \times P(\Sigma) \to \mathbb{R}$ is said to be a \emph{classical solution} to \eqref{eqn:M} 
if it is continuous in all its arguments, $C^1$ in $t$ and $C^1$ in $m$ 
and, for any $(t,x,m) \in [0,T] \times \Sigma \times P(\Sigma)$ we have
$$
\begin{cases}
& -\frac{\partial U}{\partial t}  + H(x, \Delta^x U) - \int_{\Sigma} D^m U(t,x,m,y) \cdot \alpha^*(y, \Delta^y U(t,y,m)) dm(y) = F(x,m),  \\
& U(T,x,m) = G(x,m), \ \ \ (x,m) \in \Sigma \times P(\Sigma).
\end{cases}
$$
In particular then $\Delta^x U(t,x,\cdot): P(\Sigma) \to \mathbb{R}^d$ is bounded and Lipschitz continuous and
 $D^m U(t,x,\cdot) : P(\Sigma) \to \mathbb{R}^{d\times d} $ is bounded. 

Moreover, we say that $U$ is a \emph{regular solution} to \eqref{eqn:M} if it is a classical solution and $D^m U(t,x,\cdot)$ is also Lipschitz continuous in $m$, uniformly in $(t,x)$.
\end{defn}
Let us observe that in the Master Equation we could replace $D^m U(t,x,m,y)$ by $D^m U(t,x,m,1)$, thanks to property \eqref{eqn:identity} of the derivative. Under sufficient conditions, we will prove in Section \ref{mastersection} the existence and uniqueness of a regular solution to \eqref{eqn:M}.

\section{The convergence argument}
\label{convergence}

In this section we take for granted the well-posedness of the Master Equation (\ref{eqn:M}) 
 and focus on the study of the convergence.
We give the precise statement of the convergence in terms of two theorems: the first one describes the convergence in average of the value functions, while the second one is a propagation of chaos for the optimal trajectories.

For any $i \in \left\{1,\dots, N\right\}$ and $x \in \Sigma$, set
\begin{equation*}
w^{N,i}(t_0, x, m_0) := \sum_{x_1 = 1}^d \dots \sum_{x_{i-1} =1}^d \sum_{x_{i+1} =1}^d \dots  \sum_{x_N = 1}^d v^{N,i}(t_0, \bm{x}) \prod_{j \neq i} m_0(x_j),
\end{equation*}
where $\bm{x} = (x_1, \dots, x_N)$, and 
\begin{equation*}
||w^{N,i}(t_0, \cdot, m_0) - U(t_0, \cdot, m_0)||_{L^1(m_0)} := \sum_{x=1}^d |w^{N,i}(t_0, x, m_0) - U(t_0, x, m_0)| m_0(x).
\end{equation*}
The main result is given by the following
\begin{thm}
\label{fund}
Assume \textbf{(H1)} and that (\ref{eqn:M}) admits a unique regular solution $U$ in the sense of Definition \ref{def}. 
Fix $N \geq 1$, $(t_0, m_0) \in [0,T] \times P(\Sigma)$, $\bm{x} \in \Sigma^N$  and let $(v^{N,i})$ be the solution to (\ref{eqn: HJB}). Then 
\begin{align}
\frac{1}{N} \sum_{i=1}^N |v^{N,i}(t_0, \bm{x}) - U(t_0, x_i, m_{\bm{x}}^N)| &\leq \frac CN
\label{4}\\
||w^{N,i}(t_0, \cdot, m_0) - U(t_0, \cdot, m_0)||_{L^1(m_0)} &\leq \frac{C}{\sqrt{N}}.
\label{5}
\end{align}
In \eqref{4} and \eqref{5}, the constant $C$ does not depend on $i$, $t_0$, $m_0$, $\bm{x}$ nor $N$.
\end{thm}

As stated above, the convergence can be studied also in terms of the optimal trajectories. 
Consider the optimal process 
$\bm{Y}_t = (Y_1(t),\dots,Y_N(t))_{t \in [0, T]}$ for the $N$-player game:
\begin{equation}
 Y_i(t)  = Z_i + \int_0^t\int_\Xi \sum_{y \in \Sigma} (y - Y_i(s^-)) \mathbbm{1}_{]0, \alpha_y^i(s, \bm{Y}_{s^-}) [} (\xi_y)
\mathcal{N}_i(ds,d\xi), \ \ \  t \in [0, T]
\label{eqn:13}
\end{equation}
where $\alpha^i_y(t, \bm{Y}_t)$ is the optimal feedback, i.e. 
$\alpha^i_y(t, \bm{y}) := [\alpha^*(y_i, \Delta^i v^{N,i}(t,\bm{y}))]_y$.
Moreover, let $\tilde{\bm{X}}_t = (\tilde{X}_1(t), \dots, \tilde{X}_N(t))_{t \in [0,T]}$ be the i.i.d. process solution to
\begin{equation}
\label{limit}
\tilde{X}_{i,t}  = Z_i + \int_0^t \int_\Xi \sum_{y \in \Sigma} (y - \tilde{X}_{i}(s^-)) \mathbbm{1}_{]0, \tilde{\alpha}_y^i(s, \tilde{\bm{X}}_{s^-}) [} (\xi_y)\mathcal{N}_i(ds,d\xi), \ \ \  t \in [0, T]
\end{equation}
with $\tilde{\alpha}^i_y(t, \tilde{\bm{X}}_t) := [\alpha^*\left(\tilde{X}_i(t), \Delta^x U(t, \tilde{X}_i(t), Law(\tilde{X}_i(t)))\right)]_y$.
We remark that $Law(\tilde{X}_i(t))=m(t)$ with $m$ the solution to the Mean Field Game.

\begin{thm}
\label{chaos}
Under the same assumptions of Theorem \ref{fund}, for any $N \geq 1$ and any $i \in \left\{1,\dots, N\right\}$, we have 
\begin{equation}
\mathbb{E}\left[\sup_{t\in[0,T]}\left|Y_i(t) - \tilde{X}_i(t)\right|\right] \leq CN^{- \frac{1}{9}}
\label{cha}
\end{equation}
for some constant $C > 0$ independent of $m_0$ and $N$. In particular we obtain the Law of Large Numbers
\be
\mathbb{E}\left[\sup_{t\in[0,T]}\left|m^N_{\bm{Y}}(t) - m(t)\right|\right] \leq CN^{- \frac{1}{9}}.
\label{lln}
\ee 
\end{thm}
Notice that the supremum is taken inside the mean, giving the convergence in the space of trajectories. For this reason, we have a
 slow convergence of order $N^{-1/9}$, 
coming from a result in \cite{rachev} about the convergence of the empirical measures of a decoupled system (c.f. Lemma 1 below). Instead, if the supremum were taken outside the mean, the convergence would be of order $N^{-1/2}$, thanks to a result in \cite{fournier}.

\subsection{Approximating the optimal trajectories}

The first step in the proof of these results is to show that the projection of $U$ onto empirical measures 
\begin{equation}
\label{eqn:10}
u^{N,i}(t, \bm{x}) := U(t,x_i, m_{\bm{x}}^{N,i})
\end{equation}
satisfies the system (\ref{eqn: HJB}) up to a term of order $O(\frac{1}{N})$.
The following proposition makes rigorous the intuition we already used in the heuristic derivation of the Master Equation (\ref{eqn:M}).
In what follows, $C$ will denote any constant independent of $i,N,m_0,\bm{x}$ which is allowed to change from line to line.

\begin{prop}
\label{prop1}
Let $U$ be a regular solution to the Master Equation and
$u^{N,i}(t, \bm{x})$ be defined as in (\ref{eqn:10}). Then, for $j\neq i$, 
\begin{equation}
\Delta^j u^{N,i}(t,\bm{x}) = \frac{1}{N-1} D^m U(t,x_i, m_{\bm{x}}^{N,i}, x_j) + \tau^{N,i,j}(t,\bm{x}),
\label{uni}
\end{equation}
where $\tau^{N,i,j} \in C^0([0,T] \times \Sigma^N; \mathbb{R}^d)$, $||\tau^{N,i,j}|| \leq \frac{C}{(N-1)^2}$. 
\end{prop}
\begin{proof}
Observe first that $[\Delta^j u^{N,i}(t,\bm{x})]_{x_j}=0= [D^m U(t,x_i, m_{\bm{x}}^{N,i}, x_j)] _{x_j}$ by definition, so we set 
$[\tau^{N,i,j}(t,\bm{x})]_{x_j}=0$. Consider then $h\neq x_j$: 
 $[\Delta^j u^{N,i}(t,\bm{x})]_h = U(t,x_i, \frac{1}{N-1} \sum_{k \neq i,j} \delta_{x_k} + \frac{1}{N-1}\delta_h) - U(t,x_i, m_{\bm{x}}^{N,i})$ by definition.
By standard computations we get
\begin{align*}
U & \left(t,x_i, \frac{1}{N-1} \sum_{k \neq i,j} \delta_{x_k} + \frac{1}{N-1}\delta_h \right)  - U(t,x_i, m_{\bm{x}}^{N,i}) \\
& = U \left(t,x_i,  m_{\bm{x}}^{N,i} + \frac{1}{N-1}(\delta_h - \delta_{x_j}) \right)  - U(t,x_i, m_{\bm{x}}^{N,i})\\
& = \int_{0}^{\frac{1}{N-1}}\left[D^m U(m_{\bm{x}}^{N,i} + s(\delta_h - \delta_{x_j}), x_j)\right]_h ds\\
& = \int_{0}^{\frac{1}{N-1}}\left(\left[D^m U(m_{\bm{x}}^{N,i} + s(\delta_h - \delta_{x_j}), x_j)\right]_h + \left[D^m U(m_{\bm{x}}^{N,i}, x_j)\right]_h  - \left[D^m U(m_{\bm{x}}^{N,i}, x_j)\right]_h\right) ds \\
& = \frac{1}{N-1}\left[D^m U(m_{\bm{x}}^{N,i}, x_j)\right]_h \! + \! \int_{0}^{\frac{1}{N-1}}\left(\left[D^m U(m_{\bm{x}}^{N,i} + s(\delta_h - \delta_{x_j}), x_j)\right]_h - \left[D^m U(m_{\bm{x}}^{N,i}, x_j)\right]_h\right)\! ds\\
& = \frac{1}{N-1} \left[D^m U(t,x_i, m_{\bm{x}}^{N,i}, x_j)\right]_h + O\left(\frac{1}{(N-1)^2}\right),
\end{align*}
where the last equality is derived by exploiting the Lipschitz continuity in $m$ of $D^m U$
\begin{align*}
\Bigg|\int_{0}^{\frac{1}{N-1}} &  \left(\left[D^m U(m_{\bm{x}}^{N,i} + s(\delta_h - \delta_{x_j}), x_j)\right]_h - \left[D^m U(m_{\bm{x}}^{N,i}, x_j)\right]_h\right) ds\Bigg|\\
& \leq C \int_{0}^{\frac{1}{N-1}} \left| s (\delta_h - \delta_{x_j})\right| ds = O\left(\frac{1}{(N-1)^2}\right).
\end{align*}
For every component $h$ of $D^m U$ we have found the thesis, and thus the same holds for the whole vector.
\end{proof}

 In the next proposition we show that the $u^{N,i}$'s almost solve the system (\ref{eqn: HJB}):

\begin{prop}
\label{prop3.2}
Under the assumptions of Theorem \ref{fund}, the functions $(u^{N,i})_{i = 1,\dots, N}$ solve
\begin{equation}
\label{eqn:prox}
\begin{cases}
-\frac{\partial {u}^{N,i}}{\partial t}(t,\bm{x}) \!-\! \sum_{j=1, \ j\neq i}^N\! \alpha^*(x_j, \Delta^j u^{N,j}) \cdot \Delta^j u^{N,i} \!+\! H(x_i, \Delta^i u^{N,i}) = F^{N,i}(\bm{x}) \!+\! r^{N,i}(t, \bm{x}) \\
u^{N,i}(T, \bm{x}) = G(x_i, m_{\bm{x}}^{N,i}),
\end{cases}
\end{equation}
with $r^{N,i} \in C^0([0,T] \times \Sigma^N)$, $||r^{N,i}|| \leq \frac{C}{N}$.
\end{prop}
\begin{proof}
We know that $U$ solves
\begin{equation*}
-\partial_t U + H(x, \Delta^x U) - \int_{\Sigma} D^m U (t,x,m,y)\cdot \alpha^*(y,\Delta^y U(t,y,m)) dm(y) = F(x,m),
\end{equation*}
and $U(T,x,m) = G(x,m)$.

Computing the equation in $(t,x_i, m_{\bm{x}}^{N,i})$ we get (we omit the $*$ in $\alpha^*$ for simplicity)
\begin{align*}
-\partial_t U(t,x_i,m_{\bm{x}}^{N,i}) & + H(x_i, \Delta^x U(t,x_i,m_{\bm{x}}^{N,i})) \\
& - \int_{\Sigma} D^m U(t,x_i, m_{\bm{x}}^{N,i}, y) \cdot \alpha(y,\Delta^x U(t,y,  m_{\bm{x}}^{N,i})) d m_{\bm{x}}^{N,i}(y) = F(x_i, m_{\bm{x}}^{N,i}),
\end{align*}
with the correct final condition $u^{N,i}(t,\bm{x}) = U(T,x_i,  m_{\bm{x}}^{N,i}) = G(x_i,  m_{\bm{x}}^{N,i})$.
By definition of empirical measure we can rewrite 
\begin{align*}
-\partial_t U(t,x_i,m_{\bm{x}}^{N,i}) & + H(x_i, \Delta^x U(t,x_i,m_{\bm{x}}^{N,i})) \\
& - \frac{1}{N-1}\sum_{j=1, j \neq i}^N D^m U(t,x_i, m_{\bm{x}}^{N,i}, x_j) \cdot \alpha(x_j,\Delta^x U(t,x_j,  m_{\bm{x}}^{N,i})) = F^{N,i}(\bm{x}).
\end{align*}
Thanks to Proposition \ref{prop1}, we have
\begin{align*}
\frac{1}{N-1} & \sum_{j=1, j \neq i}^N D^m U(t,x_i, m_{\bm{x}}^{N,i}, x_j) \cdot \alpha(x_j,\Delta^x U(t,x_j,  m_{\bm{x}}^{N,i})) \\
& = \sum_{j=1, j \neq i}^N \Delta^j u^{N,i}(t, \bm{x}) \cdot \alpha(x_j,\Delta^x U(t,x_j,  m_{\bm{x}}^{N,i})) \\
& -\sum_{j=1, j \neq i}^N \tau^{N,i,j}(t,\bm{x}) \cdot \alpha(x_j,\Delta^x U(t,x_j,  m_{\bm{x}}^{N,i}))\\
& =: \ 1) \  + \ 2).
\end{align*}
For the first term we add and subtract the quantity $\alpha(x_j, \Delta^x U(t,x_j, m_{\bm{x}}^{N,j}))$:
\begin{align*}
1) & = \sum_{j \neq i} \Delta^j u^{N,i}(t,\bm{x}) \cdot \alpha(x_j, \Delta^x U(t,x_j, m_{\bm{x}}^{N,i})) - \alpha(x_j, \Delta^x U(t,x_j, m_{\bm{x}}^{N,j})) \\
& + \sum_{j \neq i} \Delta^j u^{N,i}(t,\bm{x}) \cdot \alpha(x_j, \Delta^x U(t,x_j, m_{\bm{x}}^{N,j})) \\
& = (A) + (B).
\end{align*}
For $(A)$ we have, using first the Lipschitz continuity of $\alpha$ with respect to the second variable and then the Lipschitz continuity of $\Delta^x U$ with respect to $m$:
\begin{align*}
(A) & \leq \sum_{j \neq i} \Delta^j u^{N,i}(t, \bm{x}) \cdot (\Delta^x U(t,x_j, m_{\bm{x}}^{N,i}) - \Delta^x U(t,x_j, m_{\bm{x}}^{N,j})) \\
& \leq C \sum_{j \neq i} ||\Delta^j u^{N,i}|| \cdot|m_{\bm{x}}^{N,i}- m_{\bm{x}}^{N,j}| \\
& \leq \frac{C}{N-1} \sum_{j \neq i} ||\Delta^j u^{N,i}||  \leq \frac{C}{N},
\end{align*}
where the last inequality is a consequence of \eqref{uni} and the uniform bound on $||D^m U||$ for the solution to (\ref{eqn:M}). Part $(B)$ of $1)$ is instead what we want to obtain in the equation for $u^{N,i}$, so we leave it as it is.

For the term $2)$, we simply note that $\alpha$ is bounded from above by definition, and thus the whole term $2)$ is also of order $O\left(\frac{1}{N}\right)$.
\end{proof}

The central part of the proof of the convergence is based on comparing the optimal trajectories associated to $v^{N,i}$ with the ones associated to $u^{N,i}$.
Hence, 
consider the processes
\begin{equation}
\label{eqn:12}
X_i(t)  = Z_i +\int_0^t \int_\Xi \sum_{y \in \Sigma} (y - X_i(s^-)) \mathbf{1}_{]0, \tilde{\alpha}_y^i(s,\bm{X}_{s^-})[}(\xi_y) \mathcal{N}_i(ds,d\xi), \ \ \  t \in [0, T]
\end{equation}
where $\tilde{\alpha}^i_y(t, \bm{X}_t) :=[\alpha^*(X_i(t), \Delta^i u^{N,i}(t, \bm{X}_t))]_y$.
Observe that the processes $\bm{X}$ and $\bm{Y}$ are exchangeable.
For future use, let us also recall the inequalities
\be
\left|m^N_{\bm{x}} - m^N_{\bm{y}}\right| \leq C \bm{d}_1(m^N_{\bm{x}}, m^N_{\bm{y}}) \leq \frac CN \sum_{i=1}^N |x_i - y_i|
\label{memp}
\ee
for every $\bm{x},\bm{y}\in\Sigma^N$, where the first inequality comes from the equivalence of all the metrics in $P(\Sigma)$ and the second is well-known for the Wasserstein distance $\bm{d}_1$. 
The result needed to prove the main theorems is the following

\begin{thm}
\label{crucial}
With the notation introduced above, under the assumptions of Theorem \ref{fund}, we have
\begin{align}
& \mathbb{E}\left[\sup_{t \in [0, T]} |Y_i(t) - X_i(t)| \right] \leq \frac CN,\label{eqn:1}\\
& \mathbb{E}\left[\sup_{t \in [0, T]} |m^N_{\bm{Y}}(t) - m^N_{\bm{X}}(t)| \right] \leq \frac CN,\label{lxy}\\
& \mathbb{E}\left[\sup_{t \in [0, T]} |u^{N,i}(t, \bm{Y}_t) - v^{N,i}(t, \bm{Y}_t)|^2 + \int_{0}^{T}\left| \Delta^i u^{N,i}(t, \bm{Y}_t) - \Delta^i v^{N,i}(t, \bm{Y}_t)\right|^2 dt\right] \leq \frac{C}{N^{2}},\label{eqn:2}\\
&  \frac{1}{N} \sum_{i=1}^N |v^{N,i}(0, \bm{Z}) - u^{N,i}(0, \bm{Z})|\leq \frac CN \ \ \ \mathbb{P} \text{-a.s}.\label{eqn:3}
\end{align}
\end{thm}

\begin{proof}
In order to prove (\ref{eqn:2}), we apply  It\^o Formula to the function $\Psi(t, \bm{Y}_t) = (u^{N,i}(t, \bm{Y}_t) - v^{N,i}(t, \bm{Y}_t))^2$,
\begin{equation*}
d\Psi(t, \bm{Y}_t) =   \frac{\partial \Psi(t, \bm{Y}_t)}{\partial t}+ \sum_{j= 1}^N \int_\Xi [\Psi(t, \tilde{\bm{Y}}_{t^-}^j) - \Psi(t,\bm{Y}_{t-})] 
\mathcal{N}_j(dt,d\xi), 
\end{equation*}
\begin{equation*}
\tilde{\bm{Y}}_t^j = \left(Y_{1,t}, \dots, Y_{j-1,t}, Y_{j,t} + \sum_{y \in \Sigma}(y- Y_{j,t}) \mathbbm{1}_{]0, \alpha_y^j[}(\xi_y), 
Y_{j+1,t}, \dots, Y_{N,t}\right),
\end{equation*} 
and, as above,
\begin{equation*}
\alpha^j_y(t,\bm{Y}_t) = \left[\alpha^*(Y_{j,t}, \Delta^j v^{N,j}(t,\bm{Y}_t))\right]_y. 
\end{equation*}
It follows that
\begin{align*}
d\Psi(t,\bm{Y}_t) & = 2(u^{N,i}(t,\bm{Y}_t) - v^{N,i}(t,\bm{Y}_t)) (\partial_t u^{N,i} - \partial_t v^{N,i}) \\
& + \sum_{j=1}^N \int_\Xi [(u^{N,i}(t, \tilde{\bm{Y}}_{t^-}^j) - v^{N,i}(t, \tilde{\bm{Y}}_{t^-}^j))^2 - (u^{N,i}(t,\bm{Y}_{t^-}) - v^{N,i}(t, \bm{Y}_{t^-}))^2]\mathcal{N}_j(dt,d\xi).
\end{align*}
Now, integrating on the time interval $[t, T]$ we get:
\begin{align*}
[u&^{N,i}(T, \bm{Y}_T)  - v^{N,i}(T, \bm{Y}_T)]^2 =\\
& = [u^{N,i}(t, \bm{Y}_t) - v^{N,i}(t, \bm{Y}_t)]^2 + 2 \int_{t}^T (u^{N,i}(s, \bm{Y}_s) - v^{N,i}(s, \bm{Y}_s)) (\partial_t u^{N,i}(s, \bm{Y}_s) - \partial_t v^{N,i}(s, \bm{Y}_s))ds \\
& + \sum_{j=1}^N \int_{t}^T  \int_\Xi [(u^{N,i}(s, \tilde{\bm{Y}}_{s^-}^j) - v^{N,i}(s, \tilde{\bm{Y}}_{s^-}^j))^2 - (u^{N,i}(s,\bm{Y}_{s^-}) - v^{N,i}(s, \bm{Y}_{s^-}))^2]\mathcal{N}_j(ds,d\xi). \\
\end{align*}
For brevity, in the remaining part of the proof we set $u^i_t := u^{N,i}(t,\bm{Y}_t)$ and $v^i_t := v^{N,i}(t,\bm{Y}_t)$. 
Next, we take the conditional expectation on the initial data $\bm{Z}$, i.e. $\mathbb{E}^{\bm{Z}} = \mathbb{E}[\ \cdot \ |\bm{Y}_t=\bm{Z}]$;
notice that we are allowed to condition on such event since it has positive probability, thanks to the bound from below on the jump rates.
Applying again Lemma 3 of \cite{alekos}, we obtain
\begin{align*}
\mathbb{E}^{\bm{Z}}[(u^i_T  - v^i_T)^2] 
 &= \mathbb{E}^{\bm{Z}}[(u^i_t - v^i_t)^2] + 2 \mathbb{E}^{\bm{Z}}\left[ \int_{t}^T(u^i_s - v^i_s)(\partial_t u^i_s - \partial_t v^i_s)ds\right] \\
& + \sum_{j=1}^N \mathbb{E}^{\bm{Z}}\left[\int_t^{T}  \!\!\!\! \ \alpha^j(s,\bm{Y}_s) \cdot \Delta^{j}[(u^i_s - v^i_s)^2]ds\right].
\end{align*}
Let us first study the term $\mathbb{E}^{\bm{Z}}\left[ \int_{t}^T (u^i_s - v^i_s)(\partial_t u^i_s - \partial_t v^i_s)ds\right]$. Applying equations (\ref{eqn:prox}) and (\ref{eqn: HJB}) we get
\begin{align*}
 \mathbb{E}^{\bm{Z}} & \left[ \int_{t}^T(u^i_s - v^i_s)(\partial_t u^i_s - \partial_t v^i_s)ds\right] \\
 & = \mathbb{E}^{\bm{Z}}\Bigg[\int_{t}^T (u^i_s - v^i_s)\Bigg\{\sum_{j=1, j \neq i}^N \left(-\alpha^j(Y_{j,s}, \Delta^j u^j_s) \cdot \Delta^j u^i_s + \alpha^j(Y_{j,s},\Delta^j v^j_s) \cdot \Delta^j v^i_s + \alpha^j \cdot \Delta^j u^i_s \right.\\
 & \left. - \alpha^j \cdot \Delta^j u^i_s\right) - H(Y_{i,s}, \Delta^i v^i_s) + H(Y_{i,s}, \Delta^i u^i_s) - r^{N,i}(s, \bm{Y}_s)\Bigg\}ds\Bigg].
\end{align*}
Recall that $\alpha^j(\Delta^j u^j_s) =: \tilde{\alpha}^j$. Note that we also added and subtracted $ \alpha^j \cdot \Delta^j u^i_s$ in the last line so that we can use the Lipschitz properties of $H$, $\alpha^*$ and the bound on $r^{N,i}$ to get the correct estimates. 
Specifically, we can rewrite 
\begin{align*}
 \mathbb{E}^{\bm{Z}} & \left[ \int_{t}^T(u^i_s -v^i_s)(\partial_t u^i_s - \partial_t v^i_s)ds\right] \\
 & = \mathbb{E}^{\bm{Z}}\Bigg[\int_{t}^T (u^i_s - v^i_s)\Bigg\{\sum_{j=1, j \neq i}^N \left( (\alpha^j - \tilde{\alpha}^j) \cdot \Delta^j u^i_s - \alpha^j \cdot (\Delta^j u^i_s - \Delta^j v^i_s)\right) \\
 &- H(Y_{i,s}, \Delta^i v^i_s) + H(Y_{i,s}, \Delta^i u^i_s) - r^{N,i}(s, \bm{Y}_s)\Bigg\}ds\Bigg].
\end{align*}
Putting things together,
\begin{align*}
&\mathbb{E}^{\bm{Z}}[(u^i_T  - v^i_T)^2] \\
& = \mathbb{E}^{\bm{Z}}[ (u^i_t - v^i_t)^2] \! + 2 \mathbb{E}^{\bm{Z}}\!\! \left[ \int_{t}^T \! (u^i_s - v^i_s)(\partial_t u^i_s - \partial_t v^i_s)ds\right] \! + \!  \sum_{j=1}^N \mathbb{E}^{\bm{Z}}\!\! \left[\int_t^{T}  \!\!\!\!  \ \alpha^j(s,\bm{Y}_s) \! \cdot \! \Delta^{j}[(u^i_s - v^i_s)^2]ds\right] \\
& = \mathbb{E}^{\bm{Z}}[(u^i_t - v^i_t)^2] + 2  \mathbb{E}^{\bm{Z}}\Bigg[\int_{t}^T  (u^i_s - v^i_s)\Bigg\{\sum_{j \neq i}^N \left( (\alpha^j - \tilde{\alpha}^j) \cdot \Delta^j u^i_s - \alpha^j \cdot (\Delta^j u^i_s - \Delta^j v^i_s)\right)\\
 & - H(Y_{i,s}, \Delta^i v^i_s) + H(Y_{i,s}, \Delta^i u^i_s) - r^{N,i}(s, \bm{Y}_s)\Bigg\}ds\Bigg]+ \mathbb{E}^{\bm{Z}}\left[\int_t^{T}  \!\!\!\!  \ \alpha^i(s,\bm{Y}_s) \cdot \Delta^{i}[(u^i_s - v^i_s)^2]ds\right] \\
 & + \sum_{j \neq i}^N \mathbb{E}^{\bm{Z}}\left[\int_t^{T}  \!\!\!\! \ \alpha^j(s,\bm{Y}_s) \cdot \Delta^{j}[(u^i_s - v^i_s)^2]ds\right] \\
 & = \mathbb{E}^{\bm{Z}}[(u^i_t - v^i_t)^2] + 2  \mathbb{E}^{\bm{Z}}\left[\int_{t}^T  (u^i_s - v^i_s)\left\{\sum_{j \neq i}^N \left( (\alpha^j - \tilde{\alpha}^j) \cdot \Delta^j u^i_s - \alpha^j \cdot (\Delta^j u^i_s - \Delta^j v^i_s)\right)\right\} ds\right.\\
 & \left. + \int_t^{T}\sum_{j \neq i}^N \frac{1}{2}\alpha^j \cdot \Delta^{j}[(u^i_s - v^i_s)^2]ds  + \int_t^{T}\!\!(u^i_s - v^i_s)(- H(Y_{i,s}, \Delta^i v^i_s) + H(Y_{i,s}, \Delta^i u^i_s) - r^{N,i}(s, \bm{Y}_s))ds\right]\\
 &+ \mathbb{E}^{\bm{Z}}\left[\int_t^{T}  \!\!\!\!  \ \alpha^i(s,\bm{Y}_s) \cdot \Delta^{i}[(u^i_s - v^i_s)^2]ds\right]. \\
\end{align*}
On the other hand, observing that $\Delta^j [(u^i - v^i)^2] = \Delta^j (u^i-v^i) \times (\Delta^j(u^i-v^i) + 2 (\mathbf{1}(u^i-v^i)))$, $\times$ being the element by element product between vectors and $\mathbf{1} = (1, \dots, 1)^\dagger$, the expression
\begin{equation*}
\mathbb{E}^{\bm{Z}}\left[\int_{t}^T (u^i_s - v^i_s)\left\{\sum_{j=1, j \neq i}^N \left( - 2\alpha^j \cdot (\Delta^j u^i_s - \Delta^j v^i_s)\right)\right\} ds + \int_t^{T}\!\!\!\sum_{j=1, j \neq i}^N\!\! \left(\alpha^j \cdot \Delta^{j}[(u^i_s - v^i_s)^2]\right)ds \right]
\end{equation*}
can be simplified as follows
\begin{align*}
\mathbb{E}^{\bm{Z}}& \left[\int_{t}^T  (u^i_s - v^i_s)\left\{\sum_{j=1, j \neq i}^N \left( - 2\alpha^j \cdot (\Delta^j u^i_s - \Delta^j v^i_s)\right))\right\}ds + \int_t^{T}\sum_{j=1, j \neq i}^N \left(\alpha^j \cdot \Delta^{j}[(u^i_s - v^i_s)^2]\right)ds \right]\\
& = \sum_{j=1, j \neq i}^N \!\!\! \mathbb{E}^{\bm{Z}}\Bigg[\int_{t}^T\left\{ - 2\alpha^j \cdot (u^i_s - v^i_s)(\Delta^j u^i_s - \Delta^j v^i_s) \right .\\
& \left.+  \alpha^j \cdot (\Delta^j (u^i_s-v^i_s) \times (\Delta^j(u^i_s-v^i_s) + 2 (\mathbf{1}(u^i_s-v^i_s)))\right\}ds\Bigg]\\
& = \sum_{j=1, j \neq i}^N \mathbb{E}^{\bm{Z}}\left[\int_{t}^T\alpha^j \cdot (\Delta^j(u^i_s-v^i_s))^2 ds \right].
\end{align*}

Thus, we have found
\begin{align*}
0 = &\mathbb{E}^{\bm{Z}}[(u^i_T  - v^i_T)^2] \\
& = \mathbb{E}^{\bm{Z}}[(u^i_t - v^i_t)^2]  + 2  \mathbb{E}^{\bm{Z}}\Bigg[\int_{t}^T (u^i_s - v^i_s)\Bigg\{\sum_{j=1, j \neq i}^N \left( (\alpha^j - \tilde{\alpha}^j) \cdot \Delta^j u^i_s\right) \\
&- H(Y_{i,s}, \Delta^i v^i_s) + H(Y_{i,s}, \Delta^i u^i_s) - r^{N,i}(s, \bm{Y}_s)\Bigg\}ds\Bigg] + \mathbb{E}^{\bm{Z}}\left[\int_t^{T}  \!\!\!\!  \ \alpha^i(s,\bm{Y}_s) \cdot \Delta^{i}[(u^i_s - v^i_s)^2]ds\right]\\
& +  \sum_{j=1, j \neq i}^N \mathbb{E}^{\bm{Z}}\left[\int_{t}^T\alpha^j \cdot (\Delta^j(u^i_s-v^i_s))^2 ds \right].
\end{align*}
Now, using again the expression for $\Delta^i ((u^i_s - v^i_s)^2)$,
\begin{align*}
 \mathbb{E}^{\bm{Z}}&\left[\int_t^{T}  \!\!\!\!  \ \alpha^i(s,\bm{Y}_s) \cdot \Delta^{i}[(u^i_s - v^i_s)^2]ds\right] \\
 & = \mathbb{E}^{\bm{Z}}\left[\int_t^{T} \!\!\!\!  \ \alpha^i(s,\bm{Y}_s) \cdot (\Delta^i(u^i_s  - v^i_s ))^2ds\right]
  +  \mathbb{E}^{\bm{Z}}\left[\int_t^{T} \!\!\!\!  \ \alpha^i(s,\bm{Y}_s) \cdot (\Delta^i(u^i_s  -v^i_s ) \times 2 (\mathbf{1}(u^i_s  - v^i_s )ds\right],
\end{align*}
so that we can rewrite the previous as 
\begin{align*}
\mathbb{E}&^{\bm{Z}} [(u^i_t  - u^i_t)^2] + \sum_{j=1}^N \mathbb{E}^{\bm{Z}}\left[\int_{t}^T\alpha^j \cdot (\Delta^j(u^i_s -v^i_s))^2 ds \right]\\
& = - 2\mathbb{E}^{\bm{Z}}\!\!\!\left[\int_{t}^T \!\!\! (u^i_s  - v^i_s )\! \left\{\!\! \sum_{j=1, j \neq i}^N \!\!\!\!\! \left( (\alpha^j - \tilde{\alpha}^j) \cdot \Delta^j u^i_s \right)\!- \! H(Y_{i,s}, \Delta^i v^i_s) \! +\! H(Y_{i,s}, \Delta^i u^i_s)\! - \! r^{N,i}(s, \bm{Y}_s)\right\}ds\right] \\
& - \mathbb{E}^{\bm{Z}}\left[\int_t^{T} \!\!\!\!  \ \alpha^i(s,\bm{Y}_s) \cdot (\Delta^i(u^i_s - v^i_s) \times 2 (\mathbf{1}(u^i_s - v^i_s)))ds\right].
\end{align*}
Recalling that $\alpha^j \geq 0$ (since it is a vector of transition rates), we can estimate 
\begin{align*}
\mathbb{E}&^{\bm{Z}} [(u^i_t - v^i_t)^2] + \sum_{j=1}^N \mathbb{E}^{\bm{Z}}\!\!\left[\int_{t}^T\!\!\! \alpha^j \cdot (\Delta^j(u^i_s-v^i_s))^2 ds \right]\\
& \leq 2\mathbb{E}^{\bm{Z}}\!\!\! \left[\int_{t}^T \!\!\!\!\!\!  \left|u^i_s \! -  \! v^i_s\right|\left\{\sum_{j \neq i}^N \! 
\left|(\alpha^j - \tilde{\alpha}^j) \cdot \Delta^j u^i_s \right|  \! + \! \left|H(Y_{i,s}, \Delta^i v^i_s) \! -  \! H(Y_{i,s}, \Delta^i u^i_s) \right| \! + \! \left|r^{N,i}(s, \bm{Y}_s)\right|\right\}ds\right] \\
& + 2\mathbb{E}^{\bm{Z}}\!\!\left[\int_t^{T} \!\!\!\! \ \left|u^i_s - v^i_s\right| \cdot\left|\alpha^i(s,\bm{Y}_s) \cdot \Delta^i(u^i_s - v^i_s) \right|ds\right].
\end{align*}
This also implies, erasing the terms with $j \neq i$ in the left hand side,
\begin{align*}
\mathbb{E}&^{\bm{Z}} [(u^i_t - v^i_t)^2] + \mathbb{E}^{\bm{Z}}\!\!\left[\int_{t}^T \!\! \alpha^i \cdot (\Delta^i(u^i_t - v^i_t))^2 ds \right]\\
& \leq 2\mathbb{E}^{\bm{Z}}\!\!\! \left[\int_{t}^T\!\!\!\!\!  \! \left|u^i_s \! - \! v^i_s\right| \! \left\{\sum_{j \neq i}^N 
 \left|(\alpha^j - \tilde{\alpha}^j) \cdot \Delta^j u^i_s \right|  \! + \! \left|H(Y_{i,s}, \Delta^i v^i_s) \! - \! H(Y_{i,s}, \Delta^i u^i_s) \right| \! + \! \left|r^{N,i}(s, \bm{Y}_s)\right|\right\}ds\right] \\
& + 2\mathbb{E}^{\bm{Z}}\!\!\left[\int_t^{T} \!\!\!\!  \ \!\! \left|u^i_s - v^i_s\right| \left|\alpha^i(s,\bm{Y}_s) \cdot \Delta^i(u^i_s - v^i_s) \right|ds\right].
\end{align*}
For the boundedness of $\alpha^i$ from below and above (recall that the admissible controls $\alpha$ are such that $\alpha \in A = [\kappa, M]^d$), we get  
\begin{align*}
\mathbb{E}&^{\bm{Z}} [(u^i_t - v^i_t)^2] + \kappa \mathbb{E}^{\bm{Z}}\!\!\left[\int_{t}^T \left|\Delta^i(u^i_s - v^i_s)\right|^2 ds\right] \\
& \leq 2\mathbb{E}^{\bm{Z}}\!\!\left[\int_{t}^T \!\!\! \! \left|u^i_s \! -\! v^i_s\right|\left\{\!\! \sum_{j \neq i}^N \left|(\alpha^j \!- \! \tilde{\alpha}^j) \cdot \Delta^j u^i_s \right|  \! + \! \left| H(Y_{i,s}, \Delta^i v^i_s) \! - \! H(Y_{i,s}, \Delta^i u^i_s) \right|  \! + \! \left|r^{N,i}(s, \bm{Y}_s)\right|\right\}ds\right] \\
& + 2 C\mathbb{E}^{\bm{Z}}\!\! \left[\int_t^{T} \!\!\!\!  \! \left|u^i_s - v^i_s\right| \left|\Delta^i(u^i_s - v^i_s) \right|ds\right].
\end{align*}
We now use the Lipschitz continuity of $H$ and $\alpha^*$ (assumption (\textbf{H1})) and the bounds on $||r^{N,i}|| \leq \frac{C}{N}$ and $||\Delta^j u^i||\leq \frac{1}{N} ||D^m U|| \leq \frac{C}{N}$ proved in Propositions \ref{prop1} and \ref{prop3.2} to obtain
\begin{align*}
\mathbb{E}&^{\bm{Z}} [(u^i_t - v^i_t)^2] + \kappa\mathbb{E}^{\bm{Z}}\left[\int_{t}^T \left|\Delta^i(u^i_s - v^i_s)\right|^2 ds\right] \\
& \leq 2\mathbb{E}^{\bm{Z}}\left[\int_{t}^T \left|u^i_s - v^i_s\right|\left\{\frac{C}{N} \sum_{j=1, j \neq i}^N \left|\Delta^j u^j_s - \Delta^j v^j_s \right| + C \left|\Delta^i (v^i_s -  u^i_s)\right|+ \frac{C}{N}\right\}ds\right] \\
& + 2C \mathbb{E}^{\bm{Z}} \left[\int_t^{T} \!\!\!\!  \ \left|u^i_s - v^i_s\right| \left|\Delta^i(u^i_s - v^i_s) \right|ds\right]\\
& \leq \frac{C}{N} \mathbb{E}^{\bm{Z}}\!\left[\int_t^{T}\!\! |u^i_s \! - \! v^i_s|ds\right] \!+\! \frac{C}{N} \sum_{j \neq i}\mathbb{E}^{\bm{Z}}\!\left[\int_t^{T}\!\!|u_s^i \! - \! v_s^i| \left|\Delta^j(u^j_s \! - \! v^j_s)\right|ds \right]\! \\
& +\! C \mathbb{E}^{\bm{Z}}\!\left[\int_t^{T}\!\!\! |u^i_s \!-\! v^i_s| \left|\Delta^i(u^i_s \!-\! v^i_s)\right|ds\right]\!\!.
\end{align*}

By the convexity inequality $AB \leq \epsilon A^2 + \frac{B^2}{4\epsilon}$ we can further estimate the right hand side to get
\begin{align*}
\mathbb{E}&^{\bm{Z}} [(u^i_t - v^i_t)^2] + \kappa\mathbb{E}^{\bm{Z}}\left[\int_{t}^T \left|\Delta^i(u^i_s - v^i_s)\right|^2 ds\right] \\
& \leq \frac{C}{N^2} + C\mathbb{E}^{\bm{Z}}\left[\int_t^{T} |u^i_s - v^i_s|^2 ds \right] + \kappa\frac{1}{2N}\sum_{j =1}^N\mathbb{E}^{\bm{Z}}\left[\int_t^{T} \left|\Delta^j(u^j_s - v^j_s)\right|^2 ds \right].
\end{align*}
By Gronwall's Lemma, we obtain
\begin{equation}
\label{eqn:crucial}
\sup_{t \in [0,T]}\mathbb{E}^{\bm{Z}}\! [(u^i_t \!-\! v^i_t)^2] \!+\! \kappa\mathbb{E}^{\bm{Z}}\!\!\left[\int_{0}^T \left|\Delta^i(u^i_s \!-\! v^i_s)\right|^2 ds \right]\leq \frac{C}{N^2} \!+\frac{\kappa}{2N}\sum_{j =1}^N\mathbb{E}^{\bm{Z}}\!\!\left[\int_0^{T} \left|\Delta^j(u^j_s \!-\! v^j_s)\right|^2 ds \right].
\end{equation}
Taking the expectation and using the exchangeability of the processes $(Y_{j,t})_{j = 1, \dots N}$ we obtain (\ref{eqn:2}).

In order to derive (\ref{eqn:3}), we consider (\ref{eqn:crucial}) in $t= 0$ and average over $i = 1, \dots, N$, so that we get
\begin{equation*}
\frac{1}{N} \sum_{i=1}^N \mathbb{E}^{\bm{Z}}|u^{N,i}(0, \bm{Z}) - v^{N,i}(0, \bm{Z})|^2 \leq \frac{C}{N^2}, 
\end{equation*} 
which immediately implies (\ref{eqn:3}) almost surely.

We now estimate the difference $X_i - Y_i$. Thanks to equations (\ref{eqn:12}) and (\ref{eqn:13})  and the Lipschitz continuity in $x$ and $\alpha$ of the dynamics given by $f$ (see Lemma 2 in \cite{alekos}), we obtain
\begin{align}
\mathbb{E}\!\! &\left[\!\sup_{s \in [0,t]} \! |X_{i,s} - Y_{i,s}|\right]\nonumber\\ 
&\! \leq \!\! C \mathbb{E} \! \left[\int_{0}^t\!\!\! \left|\alpha^* (X_{i,s}, \Delta^i u^{N,i}(\bm{X}_s)) \! -\! \alpha^*(Y_{i,s}, \Delta^i v^{N,i}(\bm{Y}_s))\right|ds\right]\!\! + \! C \mathbb{E}\! \left[\int_{0}^t \!\!\! \left|X_{i,s} \!-\! Y_{i,s}\right|ds\right]\nonumber\\
& \leq C \mathbb{E}\left[\int_0^t \!\! |X_{i,s} - Y_{i,s}|ds\right] 
+ C\mathbb{E}\left[\int_{0}^T \!\! \left|\Delta^i u^{N,i}(\bm{Y}_s) - \Delta^i v^{N,i}(\bm{Y}_s)\right|ds\right]\nonumber\\
& + C\mathbb{E}\left[\int_{0}^t \!\! \left|\Delta^x U\left(s, X_{i,s}, m^{N,i}_{\bm{X}_s}\right) -\Delta^x U\left(s, Y_{i,s}, m^{N,i}_{\bm{Y}_s}\right)\right|ds\right]\nonumber\\
& \leq \frac CN + C \mathbb{E}\left[\int_0^t  \sup_{r \in [0,s]}|X_{i,r} - Y_{i,r}|ds\right] 
+ C\mathbb{E}\left[\int_0^t \sup_{r \in [0,s]} \left|m^{N}_{\bm{X}_r} - m^{N}_{\bm{Y}_r}\right|ds\right]\!\!,
\label{17}
\end{align} 
where we applied (\ref{eqn:2}) and the Lipschitz continuity in $m$ of $\Delta^x U$ in the last inequality.
Then inequality \eqref{memp} and the exchangeability of $(\bm{X},\bm{Y})$ yield
$$\mathbb{E}\!\! \left[\!\sup_{s \in [0,t]}|X_{i,s} - Y_{i,s}|ds\right]  \leq \frac CN + C \mathbb{E}\left[\int_0^t  \sup_{r \in [0,s]}|X_{i,r} - Y_{i,r}|ds\right]$$
and by Gronwall's inequality we get \eqref{eqn:1}.
Finally \eqref{eqn:1}, applying again \eqref{memp}, gives \eqref{lxy}.
\end{proof}


\subsection{Proofs of the main results}
We are now in the position to prove the main results.

\begin{proof}[Proof of Theorem~\ref{fund}]

For proving \eqref{4}, we just compute (\ref{eqn:3}) - which can be derived for any $t_0 \in [0,T]$ when considering processes starting from $t_0$ - for $\bm{Z}$ uniformly distributed on $\Sigma$: this yields
\begin{equation*}
\frac{1}{N} \sum_{i=1}^N |U(t_0,x_i, m_{\bm{x}}^{N,i}) - v^{N,i}(t_0, \bm{x})|\leq \frac CN.
\end{equation*}
Then, we can replace $U(t_0, x_i, m_{\bm{x}}^{N,i})$ with $U(t_0, x_i, m_{\bm{x}}^{N})$ using the Lipschitz continuity of $U$ with respect to $m$, the additional error term being of order $1/N$.


For \eqref{5}, we compute
\begin{align}
||w^{N,i}&(t_0, \cdot, m_0) - U(t_0, \cdot, m_0)||_{L^1(m_0)} = \nonumber\\ 
& = \sum_{x_i=1}^d |w^{N,i}(t_0, x_i, m_0) - U(t_0, x_i, m_0)| m_0(x_i) \nonumber \\
& = \sum_{x_i=1}^d \left| \sum_{x_1,\dots,x_{i-1},x_{i+1},\dots, x_N=1}^d v^{N,i}(t,\bm{x}) \prod_{j\neq i} m_0(x_j) - U(t,x_i,m_0)\right| m_0(x_i) \nonumber \\
& = \sum_{x_i=1}^d \left|\sum_{x_1,\dots,x_{i-1},x_{i+1},\dots, x_N=1}^d \left\{v^{N,i}(t, \bm{x}) \prod_{j \neq i} m_0(x_j) - u^{N,i}(t,\bm{x}) \prod_{j\neq i} m_0(x_j) \right.\right. \nonumber \\
& \left. \left. + u^{N,i}(t,\bm{x})\prod_{j\neq i} m_0(x_j)\right\} - U(t,x_i,m_0)\right|m_0(x_i) \nonumber \\
& \leq \mathbb{E}[|v^{N,i}(t,\bm{Z}) - u^{N,i}(t,\bm{Z})|] + \sum_{x_1,\dots,x_N=1}^d |u^{N,i}(t,\bm{x}) - U(t,x_i,m_0)|\prod_{j=1}^N m_0(x_j),
\label{20}
\end{align}
where in the last inequality the initial data $\bm{Z} = (Z_1, \dots, Z_N)$ are distributed as $m_0$.

By (\ref{eqn:2}), the first term in \eqref{20} is of order $1/N$. For the second term we further estimate, using again the Lipschitz continuity of $U$ with respect to $m$,
\begin{align*}
\sum_{x_1,\dots,x_N=1}^d & |u^{N,i}(t,\bm{x}) - U(t,x_i,m_0)|\prod_{j=1}^N m_0(x_j)\\
& =  \sum_{x_1,\dots,x_N=1}^d |U(t,x_i,m_{\bm{x}}^{N,i}) - U(t,x_i,m_0)|\prod_{j=1}^N m_0(x_j) \\
& \leq C \mathbb{E}\left[\bm{d}_1(m_{\bm{Z}}^{N,i}, m_0)\right] \leq  \frac{C}{\sqrt{N}},
\end{align*} 
where in the last inequality we used that $\mathbb{E}\left[\bm{d}_1(m_{\bm{Z}}^{N}, m_0)\right] \leq \frac{C}{\sqrt{N}}$, thanks to Theorem 1 of
\cite{fournier}, where $\bm{Z} := (Z_1 \dots, Z_N)$, the $Z_i$'s are i.i.d. initial data, $m_0$-distributed, $\bm{d}_1$ is the $1$-Wasserstein distance and $m_{\bm{Z}}^N$ is the corresponding empirical measure.  
Overall, we have bounded \eqref{20} by a term of order $1/\sqrt{N}$, and thus  \eqref{5} is also proved.
\end{proof}

Finally, we get to the proof of the propagation of chaos (Theorem \ref{chaos}). Recall that the $Y_{i,t}$'s are the optimal processes, i.e. the solutions to system 
\eqref{eqn:13}, the $X_{i,t}$'s are the processes associated to the functions $u^{N,i}$, i.e. they solve system \eqref{eqn:12}, while the $\tilde{X}_{i,t}$'s - to which we would like to prove convergence - are the decoupled limit processes (they solve system \eqref{limit}). First, we need the following lemma, whose proof can be found for example in \cite{rachev}:
\begin{lem}
\label{lemma_imp}
Let $\tilde{\bm{X}}_t = (\tilde{X}_{i,t})_{i \in {1,\dots,N}}$ be $N$ i.i.d. processes with values in $\mathbb{R}$, 
with $Law(\tilde{X}_{i,t}) = m(t)$.
Then
\begin{equation}
\label{lem_chaos}
\mathbb{E}\left[\sup_{t \in [0,T]} \left|m_{\tilde{\bm{X}}_t}^{N,i} - m_t\right|\right] \leq C \mathbb{E}\left[\sup_{t \in [t_0,T]} \bm{d}_1 (m_{\tilde{\bm{X}}_t}^{N,i}, m_t)\right] \leq C N^{-1/9}.
\end{equation}
\end{lem} 

\begin{proof}[Proof of Theorem~\ref{chaos}] 

\noindent
The assertion of the theorem is proved if we show that
\begin{equation}
\label{last}
\mathbb{E}\left[\sup_{t \in [0,T]} |X_{i,t} - \tilde{X}_{i,t}|\right] \leq C N ^{-1/9}.
\end{equation}
Indeed, by the triangle inequality and \eqref{eqn:1} in Theorem \ref{crucial} we can estimate
\begin{align*}
\mathbb{E}\left[\sup_{t \in [0,T]} |Y_{i,t} - \tilde{X}_{i,t}|\right] & \leq \mathbb{E}\left[\sup_{t \in [0,T]} |Y_{i,t} - X_{i,t}|\right]  + \mathbb{E}\left[\sup_{t \in [0,T]}|X_{i,t} - \tilde{X}_{i,t}|\right]\\
& \leq C(N^{-1} + N^{-1/9}).
\end{align*}
We are then left to prove \eqref{last}. As in the proof of \eqref{eqn:1}, we have
\begin{align*}
\rho&(t) := \mathbb{E}\left[\sup_{s \in [0,t]} |X_{i,s} - \tilde{X}_{i,s}|\right] \\
& \leq \mathbb{E}\left[\int_{0}^t \left|\alpha^*(X_{i,s}, \Delta^i u^{N,i}(\bm{X}_s)) - 
\alpha^*( \tilde{X}_{i,s}, \Delta^x U(s, \tilde{X}_{i,s}, m(s)))\right| ds + \int_{0}^t \left|X_{i,s} - \tilde{X}_{i,s}\right|ds\right]\\
& \leq \mathbb{E}\left[\int_{0}^t \left| \alpha^*(X_{i,s}, \Delta^x U(r,X_{i,s},m_{\bm{X}_s}^{N,i})) - \alpha^*(X_{i,s}, \Delta^x U(s,\tilde{X}_{i,s},m_{\tilde{\bm{X}}_s}^{N,i}))\right| ds + \int_{0}^t \left|X_{i,s} - \tilde{X}_{i,s}\right|ds\right.\\
& \left.+ \int_{0}^t \left|\alpha^*(X_{i,s}, \Delta^x U(s,\tilde{X}_{i,s},m_{\tilde{\bm{X}}_s}^{N,i})) - \alpha^*(\tilde{X}_{i,s}, \Delta^x U(s,\tilde{X}_{i,s},m(s)))\right| ds\right].
\end{align*}
By the Lipschitz continuity of the optimal controls, and of $\Delta^x U$, we can write
\begin{align*}
\rho(t) & \leq C \int_{0}^t  \mathbb{E}\left[|X_{i,s} - \tilde{X}_{i,s}| + \left|m_{\tilde{\bm{X}}_s}^{N,i} - m_{\bm{X}_s}^{N,i}\right| + \left|m_{\tilde{\bm{X}}_s}^{N,i} - m(s)\right|\right]ds\\
& \leq C \int_{0}^t \mathbb{E} \left[|X_{i,s} - \tilde{X}_{i,s}| + \frac{1}{N-1}\sum_{j \neq i} |X_{j,s} - \tilde{X}_{j,s}| + \left|m_{\tilde{\bm{X}}_s}^{N,i} - m(s)\right|\right]ds.
\end{align*}
Using \eqref{lem_chaos} of Lemma \ref{lemma_imp} and the exchangeability of the processes, we obtain
\begin{align*}
\rho(t) &  
\leq C \int_{0}^t \left(\mathbb{E}\left[\sup_{r \in [0,s]} |X_{i,r} - \tilde{X}_{i,r}|\right] + \frac{1}{N-1}\sum_{j \neq i}\mathbb{E}\left[\sup_{r \in [0,s]} |X_{j,r} - \tilde{X}_{j,r}|\right]\right) ds \\
& + C \mathbb{E}\left[\sup_{r \in [0,T]}\left|m_{\tilde{\bm{X}}_r}^{N,i} - m_r\right|\right]\\
& \leq C \int_{0}^t \rho(s) ds + C N^{-1/9},
\end{align*}
which, by Gronwall's Lemma, concludes the proof of \eqref{cha}. Finally \eqref{lln} follows from \eqref{cha} and \eqref{lem_chaos}, using also 
\eqref{memp}.
\end{proof}

\section{Fluctuations and Large Deviations}
\label{fl}
The convergence results, Theorem \ref{fund} and \ref{chaos}, allow to derive a Central Limit Theorem and a Large Deviation Principle for the asymptotic behaviour of the empirical measure process of the $N$-player game  optimal trajectories. First of all, we recall from Proposition \ref{simmetry} that, for any $i$,
 the value function $v^{N,i}$ of player $i$ in the $N$-player game is invariant under permutations of
 $(x_1,\ldots,x_{i-1},x_{i+1}, \ldots,x_N)$. This is equivalent to say that the value functions can be viewed as functions of the empirical measure of the system,
i.e. there exists a map $V^N: [0,T] \times \Sigma \times P(\Sigma)$ such that
\be
v^{N,i}(t, \bm{x}) = V^N (t,x_i, m_{\bm{x}}^{N,i})
\ee
for any $i=1,\ldots,N$, $t\in [0,T]$ and $\bm{x}\in \Sigma^N$.

\subsection{Dynamics of the empirical measure process}

We consider the empirical measure process of the optimal evolution $\bm{Y}$ - defined in \eqref{eqn:13} - of the $N$-player game. If the system is in $\bm{x}$ at time $t$, then the rate at which player $i$ goes from $x_i$ to $y$ is given, via the optimal control, by 
\be
\alpha^*_y(x_i, \Delta^i V^N (t,x_i, m_{\bm{x}}^{N,i})) =: \Gamma^N _{x_i,y}(t, m^N_{\bm{x}}),
\ee
i.e. by a function $\Gamma^N$ which depends only on the empirical measure  $m^N_{\bm{x}}$ and on the number of players $N$. 

Thus the empirical measure of the system $(m^N_t)_{t\in [0,T]}$, 
$m^N_t := m^N_{\bm{Y}}(t)= \frac 1N \sum_{i=1}^N \delta_{Y_{i,t}}$,
 evolves as a (time-inhomogeneous) Markov process on $[0,T]$, with values in 
$S_N := P(\Sigma) \cap \frac 1N \mathbb{Z}^d$. The number of players in state $x$, when the empirical measure is $m$, is $N m_x$. 
Hence the jump rate of $m^N$ in the direction 
$\frac 1N (\delta_y - \delta_x)$ at time $t$ is $N m_x \Gamma^N_{x,y} (t,m)$. 
Therefore the generator of the time-inhomogeneous Markov process $m^N$ is given, at time $t$, by
\be
\mathcal{L}^N_t g (m) := N \sum_{x,y \in \Sigma} m_x \Gamma^N_{x,y} (t,m) \left[g\left(m+\frac 1N (\delta_y - \delta_x) \right) -g(m)\right],
\label{mn}
\ee
for any $g: S_N \longrightarrow \mathbb{R}$.

Theorem \ref{chaos} implies that the empirical measures converge in $L^1$ 
- on the space of trajectories $D([0,T],P(\Sigma))$ - to the deterministic flow of measures $m$ which is the unique solution of the Mean Field Game system, whose dynamics is given by the KFP ODE
\be
\begin{cases}
\frac{d}{dt} m(t) = \Gamma (t,m(t))^\dagger m(t)\\
m(0) = m_{0},
\end{cases}
\ee
where $\Gamma$ is the matrix defined by 
\be
\Gamma_{x,y}(t,m):= \alpha^{*}_y(x, \Delta^x U(t,x,m))
\label{Gamma}
\ee
 and $U$ is the solution to the Master Equation. 
Viewing $m(t)$ as a Markov process - and so we will  write $m_t$ in this section -, its infinitesimal generator is given, at time $t$, by
\be
\mathcal{L}_t g (m) :=  \sum_{x,y \in \Sigma} m_x \Gamma_{x,y} (t,m) [D^m g(m,x)]_y
\ee
for any $g: P(\Sigma) \longrightarrow \mathbb{R}$. Thanks to \eqref{eqn:identity}, the generator can be equivalently written as 
\be
\mathcal{L}_t g (m) := \sum_{x,y \in \Sigma} m_x \Gamma_{x,y} (t,m) [D^m g(m,1)]_y = m^\dagger \Gamma(t,m) D^m g(m,1).
\ee

In order to prove the asymptotic results, we will also consider the empirical measure of the process $\bm{X}$ defined in \eqref{eqn:12}, in which each player chooses the same control 
$\Gamma_{x,y}$ independent of $N$. We denote by $\eta^N_t := \frac 1N\sum_{i =1}^N \delta_{X_i(t)}$ the empirical measure process of $\bm{X}$, whose generator is given, for any $g: P(\Sigma) \longrightarrow \mathbb{R}$, by
\be
\mathcal{M}_t^N g(m) := N \sum_{x,y \in \Sigma} m_x \Gamma_{x,y}(t,m)\left[g\left(m + \frac{1}{N}(\delta_y - \delta_x)\right) - g(m)\right].
\label{eta}
\ee

\subsection{Central Limit Theorem}

A natural refinement of the Law of Large Numbers \eqref{lln} consists in studying the fluctuations around the limit,
 that is the asymptotic distribution of $m^N_t - m_t$. 

This can be done through a functional Central Limit Theorem: we define the fluctuation flow
\begin{equation}
\rho^N_t := \sqrt{N}(m^N_t - m_t) \quad t \in [0,T]
\label{eqn:fluctuation}
\end{equation}
and study its asymptotic behavior as $N$ tends to infinity. We follow a classical weak convergence approach based on uniform convergence of the generator of the fluctuation flow \eqref{eqn:fluctuation} to a limiting generator of a diffusion process to be determined; see  e.g. \cite{daipra} for reference. Before stating the theorem we observe that the process \eqref{eqn:fluctuation} has values in $P_0(\Sigma)$, which in the following we treat as a subset of $\mathbb{R}^{d}$. 
\begin{thm}[Central Limit Theorem]
Let $U$ be a regular solution to the Master Equation and assume (\textbf{H1}). Then the fluctuation flow $\rho^N_t$ in \eqref{eqn:fluctuation} converges, as $N \to \infty$, in the sense of weak convergence of stochastic processes, to a limiting Gaussian process $\rho_t$ which is the solution of the linear SDE
\begin{equation}
\begin{cases}
\label{eqn:diff}
d \rho_t = \left(\Gamma(t, m_t)^\dagger \rho_t + b(t, m_t, \rho_t)\right) dt + \sigma(t,m_t) dB_t,\\
\rho_0= \bar{\rho},
\end{cases}
\end{equation}
where $\bar{\rho}$ is the limit of $\rho^N_0$ in distribution, $B$ is a standard $d$-dimensional Brownian motion, $\Gamma$ is the transition rate matrix  in \eqref{Gamma}, $b\in \mathbb{R}^d$ is linear in $\mu$ and defined, for any $y \in \Sigma$ and  $\mu \in P_0(\Sigma)$, by
\begin{equation}
b(t, m, \mu)_y := \sum_{x \in \Sigma} m_x \left[D^m \Gamma_{x,y}(t, m,1) \cdot \mu\right],
\end{equation}
 and $\sigma \in \mathbb{R}^{d\times d}$ is given by the relations
\begin{align}
\label{eqn:relations}
(\sigma^2)_{x,y}(t,m) & = -(m_x \Gamma_{x,y}(t,m) + m_y \Gamma_{y,x}(t,m)), \text{ for } x \neq y,\\
\label{eqn:relations2}
(\sigma^2)_{x,x}(t,m) & = \sum_{y \neq x} (m_y \Gamma_{y,x}(t,m) + m_x \Gamma_{x,y}(t,m)).
\end{align}
In particular the matrix $\sigma^2$ is the opposite of the generator of a Markov chain, is symmetric and positive semidefinite with one null eigenvalue, and the same properties hold for $\sigma$, meaning that $\rho_t \in P_0(\Sigma)$ for any $t$.
\label{clt}
\end{thm}
\begin{proof}
The key observation is that we can reduce ourselves to study the asymptotics of the fluctuation flow
\begin{equation}
\label{eqn:new_flow}
\mu^N_t := \sqrt{N}(\eta^N_t - m_t),
\end{equation}
which is more standard since $\eta^N_t$, whose generator $\mathcal{M}$ is defined in  \eqref{eta}, is the empirical measure of an uncontrolled system of $N$ mean-field interacting particles.
Indeed, by  \eqref{lxy} we have that $\sqrt{N}(m^N - \eta^N)$ tends to $0$ almost surely as $N$ goes to infinity. 

Thus, it remains to prove the convergence in law of \eqref{eqn:new_flow} to the solution to \eqref{eqn:diff}. 
The convergence of $\mu^N_0$ (and $\rho^N_0$) to the initial condition $\bar{\rho}$ follows from the Central Limit Theorem for the i.i.d. sequence of initial conditions $Z_i$ in systems (22) and (28).  
Then, we compute the generator of \eqref{eqn:new_flow} for $t\geq 0$.
We note that $\mu^N_t$ is obtained from $\eta^N_t$ through a time dependent, linear invertible transformation $\Phi_t : S_N \to P_0(\Sigma) \subset \mathbb{R}^{d}$, defined by
\begin{equation*}
\Phi_t(\vartheta) := \sqrt{N}(\vartheta - m_t),
\end{equation*}
with inverse $\Phi_t^{-1}(\mu) := m_t + \frac{\mu}{\sqrt{N}}$.
Thus, the generator $\mathcal{H}_t^N$ of \eqref{eqn:new_flow} can be written as
\begin{equation}
\label{eqn:generator}
\mathcal{H}_t^N g (\mu) = \mathcal{M}_t^N[g \circ \Phi_t](\Phi_t^{-1}(\mu)) + \frac{\partial}{\partial t} [g \circ \Phi_t](\Phi_t^{-1}(\mu)),
\end{equation}
for any $g : P_0(\Sigma) \to \mathbb{R}$ regular and with compact support (we can extend the definition of $g$ to be a smooth function in the whole space $\mathbb{R}^d$, so that the usual derivatives are well defined).
We have
\begin{align*}
\frac{\partial}{\partial t} [g \circ \Phi_t](\Phi_t^{-1}(\mu)) & = - \sqrt{N} \nabla_{\mu} g(\mu) \cdot \frac{d}{dt}m_t = - \sqrt{N} \nabla_{\mu} g(\mu) \cdot \left(\Gamma\left(t,m_t\right)^\dagger m_t\right)\\
& = -\sqrt{N}\sum_{x,y \in \Sigma} \frac{\partial}{\partial \mu_y}g(\mu)\Gamma_{x,y}(t,m_t)(m_{t})_x.
\end{align*}
where the second equality follows from the KFP equation for $m_t$.
For the remaining part in  \eqref{eqn:generator}, we have
\begin{align*}
\mathcal{M}_t^N[g \circ \Phi_t](\Phi_t^{-1}(\mu)) & = N \sum_{x,y \in \Sigma} \left(m_t + \frac{\mu}{\sqrt{N}}\right)_x \Gamma_{x,y}\left(t, m_t + \frac{\mu}{\sqrt{N}}\right)\times\\
& \times \left\{[g\circ \Phi_t]\left(m_t + \frac{\mu}{\sqrt{N}} + \frac{1}{N}(\delta_y - \delta_x)\right) - [g \circ \Phi_t]\left(m_t + \frac{\mu}{\sqrt{N}}\right)\right\}\\
& =   N \!\! \sum_{x,y \in \Sigma} \!\! \left(m_t \!+\! \frac{\mu}{\sqrt{N}}\right)_x \!\! \Gamma_{x,y}\!\!\left(t, m_t \!+\! \frac{\mu}{\sqrt{N}}\right)\left\{g\left(\mu + \frac{1}{\sqrt{N}}(\delta_y - \delta_x)\right) - g(\mu)\right\}.
\end{align*}
Thus, we have found
\begin{align*}
\mathcal{H}_t^N g (\mu) & = N \sum_{x,y \in \Sigma} \left(m_t + \frac{\mu}{\sqrt{N}}\right)_x \Gamma_{x,y}\left(t, m_t + \frac{\mu}{\sqrt{N}}\right)\left\{g\left(\mu + \frac{1}{\sqrt{N}}(\delta_y - \delta_x)\right) - g(\mu)\right\}\\
& -\sqrt{N}\sum_{x,y \in \Sigma} \frac{\partial}{\partial \mu_y}g(\mu)\Gamma_{x,y}(t,m_t)(m_{t})_x.
\end{align*}
In order to perform a Taylor expansion of the generator, we first develop the term
\begin{align*}
g\left(\mu + \frac{1}{\sqrt{N}}(\delta_y - \delta_x)\right) &- g(\mu)\\
& = \frac{1}{\sqrt{N}} \nabla_{\mu}g(\mu) \cdot (\delta_y - \delta_x) + \frac{1}{2N} (\delta_y - \delta_x)^\dagger D^2_{\mu\mu} g(\mu)(\delta_y - \delta_x) + O\left(\frac{1}{N^{3/2}}\right)\!.
\end{align*}
Substituting, we get
\begin{align*}
\mathcal{H}_t^N g (\mu) & = \sqrt{N} \sum_{x,y \in \Sigma}  \left(m_t + \frac{\mu}{\sqrt{N}}\right)_x  \Gamma_{x,y}\left(t, m_t + \frac{\mu}{\sqrt{N}}\right) \nabla_{\mu} g(\mu) \cdot (\delta_y - \delta_x) \\
& + \frac{1}{2} \sum_{x,y \in \Sigma}  \left(m_t + \frac{\mu}{\sqrt{N}}\right)_x  \Gamma_{x,y}\left(t, m_t + \frac{\mu}{\sqrt{N}}\right) (\delta_y - \delta_x)^\dagger D^2_{\mu\mu}g(\mu) (\delta_y - \delta_x)\\
& -\sqrt{N}\sum_{x,y \in \Sigma} \frac{\partial}{\partial \mu_y}g(\mu)\Gamma_{x,y}(t,m_t)(m_{t})_x + O\left(\frac{1}{\sqrt{N}}\right).
\end{align*}
Now, we note that 
\begin{align*}
\sum_{x,y \in \Sigma} &  \left(m_t + \frac{\mu}{\sqrt{N}}\right)_x  \Gamma_{x,y}\left(t, m_t + \frac{\mu}{\sqrt{N}}\right) \nabla_{\mu} g(\mu) \cdot (\delta_y - \delta_x)\\
&= \sum_{x,y \in \Sigma}  \left(m_t + \frac{\mu}{\sqrt{N}}\right)_x  \Gamma_{x,y}\left(t, m_t + \frac{\mu}{\sqrt{N}}\right) \frac{\partial}{\partial \mu_y} g(\mu),
\end{align*}
since $\sum_y \Gamma_{x,y} = 0$.
This property allows us to rewrite
\begin{align*}
\mathcal{H}_t^N g (\mu) & = \sum_{x,y \in \Sigma} \mu_x \Gamma_{x,y}\left(t, m_t + \frac{\mu}{\sqrt{N}}\right) \frac{\partial}{\partial \mu_y} g(\mu) \\
&+ \sqrt{N}\sum_{x,y \in \Sigma} (m_{t})_x\frac{\partial}{\partial \mu_y} g(\mu) \left[\Gamma_{x,y}\left(t,m_t + \frac{\mu}{\sqrt{N}}\right) - \Gamma_{x,y}(t,m_t)\right]  \\
& + \frac{1}{2} \sum_{x,y \in \Sigma}  \left(m_t + \frac{\mu}{\sqrt{N}}\right)_x  \Gamma_{x,y}\left(t, m_t + \frac{\mu}{\sqrt{N}}\right) (\delta_y - \delta_x)^\dagger D^2_{\mu\mu}g(\mu) (\delta_y - \delta_x) + O\left(\frac{1}{\sqrt{N}}\right).
\end{align*} 
Then, using the Lipschitz continuity of $\Gamma$ as we did in Proposition 3, we linearize the term
\begin{equation*}
\Gamma_{x,y}\left(t,m_t + \frac{\mu}{\sqrt{N}}\right) - \Gamma_{x,y}(t,m_t) = \frac{1}{\sqrt{N}} D^m\Gamma_{x,y}(t,m_t,1) \cdot \mu + O\left(\frac{1}{N}\right).
\end{equation*}
We thus deduce that
\begin{equation*}
\lim_{N \to \infty} \!\sup_{t \in [0,T]}\sup_{\mu \in P_0(\Sigma)} |\mathcal{H}^N_t g(\mu) -  \mathcal{H}_t g(\mu) | = 0 
\end{equation*}
for any $g$, the convergence being of order $\frac{1}{\sqrt{N}}$,
where 
\begin{align}
\label{eqn:limit_gen}
\mathcal{H}_t g (\mu) := \sum_{x,y \in \Sigma} & \mu_x \Gamma_{x,y}\left(t, m_t\right) \frac{\partial}{\partial \mu_y} g(\mu) + \sum_{x,y \in \Sigma}  (m_{t})_x \left[ D^m\Gamma_{x,y}(t,m_t,1) \cdot \mu\right]\frac{\partial}{\partial \mu_y} g(\mu)\\
& + \frac{1}{2} \sum_{x,y \in \Sigma}  \left(m_t\right)_x  \Gamma_{x,y}\left(t, m_t \right) (\delta_y - \delta_x)^\dagger D^2_{\mu\mu}g(\mu) (\delta_y - \delta_x).\nonumber
\end{align}
The proof is then completed if we show that the generator \eqref{eqn:limit_gen} is associated to the SDE \eqref{eqn:diff}.

The drift component can be immediately identified, since
\begin{equation*}
\sum_{x,y \in \Sigma} \mu_x \Gamma_{x,y}\left(t, m_t\right) \frac{\partial}{\partial \mu_y} g(\mu) = 
\left(\Gamma(t,m_t)^\dagger \mu \right)\cdot \nabla_{\mu} g(\mu),
\end{equation*}
and 
\begin{equation*}
\sum_{x,y \in \Sigma} (m_{t})_x \left[ D^m\Gamma_{x,y}(t,m_t,1) \cdot \mu\right]\frac{\partial}{\partial \mu_y} g(\mu) = b(t,\mu) \cdot \nabla_{\mu}g(\mu).
\end{equation*}
For the diffusion component, we first note that, for each $x, y \in \Sigma$,
\begin{equation*}
(\delta_y - \delta_x)^\dagger D^2_{\mu \mu}g(\mu) (\delta_y - \delta_x)  = \frac{\partial^2}{\partial \mu_y \mu_y} g(\mu) + \frac{\partial^2}{\partial \mu_x \mu_x} g(\mu) - \frac{\partial^2}{\partial \mu_x \mu_y} g(\mu) - \frac{\partial^2}{\partial \mu_y \mu_x} g(\mu),
\end{equation*}
so that
\begin{align*}
 \frac{1}{2} \sum_{x,y \in \Sigma}  & (\delta_y - \delta_x)^\dagger D^2_{\mu\mu}g(\mu) (\delta_y - \delta_x) (m_t)_x  \Gamma_{x,y}(t,m_t)  \\
 & = \frac{1}{2} \sum_{x,y \in \Sigma} \left[\frac{\partial^2}{\partial \mu_y \mu_y} g(\mu) + \frac{\partial^2}{\partial \mu_x \mu_x} g(\mu) - \frac{\partial^2}{\partial \mu_x \mu_y} g(\mu) - \frac{\partial^2}{\partial \mu_y \mu_x} g(\mu)\right] (m_t)_x  \Gamma_{x,y}(t,m_t),
\end{align*}
which is equal to  
\begin{equation*}
\frac{1}{2}Tr(\sigma^2 (t,m_t)D^2_{\mu\mu}g(\mu)) = \frac{1}{2}\sum_{x,y\in \Sigma} (\sigma^2(t,m_t))_{x,y} \frac{\partial^2}{\partial \mu_x \mu_y} g(\mu),
\end{equation*}
 if we define $(\sigma^2)_{x,y \in \Sigma}$ by the relations \eqref{eqn:relations} and \eqref{eqn:relations2}.

Finally, we observe that the limiting process $\rho_t$ defined in \eqref{eqn:diff} takes values in $P_0(\Sigma)$, as required. Indeed, by diagonalizing $\sigma^2$ - which is symmetric and such that its rows sum to $0$ -  we get that all the eigenvectors, besides the constant one relative to the null eigenvalue, have components which sum to $0$ (by orthogonality). The same properties hold for the square root matrix $\sigma$, so that equation \eqref{eqn:diff} preserves the space $P_0(\Sigma)$.     
\end{proof}

\subsection{Large Deviation Principle}

We state the large deviation result, which is a sample path Large Deviation Principle on $D([0,T];P(\Sigma))$. To define the rate function, we first introduce the local rate function $\lambda:\mathbb{R}\longrightarrow [0,+\infty]$,
\be
\lambda(r):=
\begin{cases}
r \log r - r +1 \quad& r>0,\\
1 & r=0,\\
+\infty & r<0.
\end{cases}
\ee
For $t\in [0,T]$, $m\in P(\Sigma)$ and $\mu \in P_0(\Sigma)$, define
\be
\Lambda(t, m,\mu) := \inf \left\{\sum_{x,y \in \Sigma} m_x\Gamma_{x,y}(t,m) \lambda\left(\frac{q_{x,y}}{\Gamma_{x,y}(t,m)}\right) :
q_{x,y} \geq 0 , \sum_{x,y\in\Sigma} q_{x,y} (\delta_y - \delta_x) = \mu \quad \forall x,y\right\}
\ee
and set, for $\gamma:[0,T]\longrightarrow P(\Sigma)$,
\be
I(\gamma) := 
\begin{cases}
\int_0^T \Lambda(t, \gamma(t),\dot{\gamma}(t))dt \quad & \mbox{ if } \gamma \mbox{ is absolutely continuous and } \gamma(0)=m_0 \\
+\infty &\mbox{ otherwise.}
\end{cases}
\label{3.8}
\ee

We are now able to state the Large Deviation Principle. We equip $D([0,T];P(\Sigma))$ with the Skorokhod $J_1$-topology and denote by 
$\mathcal{B}(D([0,T];P(\Sigma)))$ the associated Borel $\sigma$-algebra.

\begin{thm}[Large Deviation Principle]
\label{LD}
Let $U$ be a regular solution to the Master Equation and assume (\textbf{H1}).
Also, assume  that the initial conditions $(m^N_0)_{N\in\mathbb{N}}$ are deterministic and $\lim_N m^N_0=m_0$. Then the sequence of empirical measure processes $(m^N)_{N\in\mathbb{N}}$ satisfies the sample path Large Deviation Principle on $D([0,T];P(\Sigma))$ with the (good) rate function $I$. Specifically, 
\begin{itemize}
	\item[(i)] if $E\in \mathcal{B}(D([0,T];P(\Sigma)))$ is closed then
	\be
	\limsup_N \frac 1N \log \mathbb{P}(m^N \in E) \leq - \inf_{\gamma\in E} \left\{I(\gamma)  \right\}
	\ee
	\item[(ii)]  if $E\in \mathcal{B}(D([0,T];P(\Sigma)))$ is open then
	\be
	\liminf_N \frac 1N \log \mathbb{P}(m^N \in E) \geq - \inf_{\gamma\in E} \left\{I(\gamma)  \right\}
	\ee
	\item[(iii)] For any  $M<+\infty$ the set 
	\be
	\left\{\gamma \in D([0,T];P(\Sigma)) : I(\gamma) \leq M \right\}
	\ee
	is compact.
\end{itemize}
\label{large}
\end{thm}

We remark that the initial conditions are assumed to be deterministic only for simplicity, otherwise there would be another term in the rate function $I$.
Before proving Theorem \ref{LD}, let us give another characterization of $I$. For $m\in P(\Sigma)$ and 
$\theta\in \mathbb{R}^d$, define
\be
\Psi(t,m,\theta):= \sum_{x,y} m_x\Gamma_{x,y}(t,m) \left[e^{\theta\cdot(\delta_y-\delta_x)} -1\right]
\ee
and let $\Lambda^0$ be the Legendre transform of $\Psi$:
\be
\Lambda^0(t,m,\mu)= \sup_{\theta\in\mathbb{R}^d} \left[\theta\cdot\mu - \Psi(t,m,\theta)\right].
\ee
Define $I^0$ as in \eqref{3.8} but with $\Lambda$ replaced by $\Lambda^0$. Via a standard result in convex analysis, Proposition 6.2 in \cite{DRW} shows that 
$\Lambda=\Lambda^0$ and then $I=I^0$.


Several authors studied large deviation properties of mean field interacting processes similar to ours. However, most of them deal with the case in which the prelimit jump rates, $m^N_x\Gamma^N$, are constant and equal to the limiting rates $m_x\Gamma$; see e.g. \cite{25}, \cite{32} and \cite{31}. We mention that in this latter paper, as in many others, it is also assumed that the jump rates of the prelimit process are bounded from below and away from 0; this does not apply to our case, since the number of agents in a state $x$ could be 0, implying that $m^N_x \Gamma^N_{x,y}$ might also be 0. 

To prove the claim, we apply the results in \cite{DRW}: to our knowledge, it is the first paper which proves a Large Deviation Principle considering the jump rates of any player depending on $N$ (and deals also with systems with simultaneous jumps).  Theorem 3.4.1 in \cite{35} shows, however, the exponential equivalence of the processes $m^N$ and the processes $\eta^N$ given by \eqref{eta} in which the jump rates of the prelimit system $m^N_x\Gamma^N$ are replaced by $m_x\Gamma$, which does not depend on $N$; the proof uses a coupling of the two Markov chains. 
The results in \cite{DRW} and \cite{35} are derived assuming  the following properties:
\begin{enumerate}
	\item the dynamics of any agent is ergodic and the jump rates are uniformly bounded;
	\item for each $x,y\in \Sigma$, the limiting jump rates $\Gamma_{x,y}$ are Lipschitz continuous in $m$;
	\item for each $x,y\in \Sigma$, given any sequence $m^N\in S_N$ such that $\lim_N m^N=m$, 
	\be
\lim_N \sup_{0\leq t\leq T} |m^N_x\Gamma^N_{x,y}(t,m^N) - m_x\Gamma_{x,y}(t,m)|=0.
\label{2.3}
\ee
\end{enumerate}
Property (1) holds in our model since the jump rates of any player belong to $[\kappa,M]$, while (2) is true because of the regularity of the solution $U$ to the Master Equation.

\begin{proof}[Proof of Theorem \ref{large}]

The fact that $I$ is a good rate function, i.e condition (iii), is proved for instance in Theorem 1.1 of \cite{12}.
Due to Theorem 3.9 in  \cite{DRW}, in order to prove the claims (i) and (ii), it is enough to show  \eqref{2.3}. Actually \cite{DRW} studies time homogeneous Markov processes, but their results still apply in the non-homogeneous case if one proves the uniform in time convergence given by \eqref{2.3}.

Let $x,y\in\Sigma$, $m^N=m^N_{\bm{x}} \in S_N$, $\bm{x}= (x_1,\ldots,x_N)\in\Sigma^N$ 
and $m^N_{\bm{x}}\rightarrow m$. Then
\begin{align*}
\big|[m^N_{\bm{x}}]_x \Gamma^N_{x,y}(t,m^N_{\bm{x}})& - m_x\Gamma_{x,y}(t,m)\big| \leq \big|[m^N_{\bm{x}}]_x\Gamma^N_{x,y}(t,m^N_{\bm{x}}) - [m^N_{\bm{x}}]_x\Gamma_{x,y}(t,m^N_{\bm{x}})\big|\\
& + \big|[m^N_{\bm{x}}]_x\Gamma_{x,y}(t,m^N_{\bm{x}}) - m_x\Gamma_{x,y}(t,m)\big| =: A + B.
\end{align*}
The first term goes to zero, uniformly over time, thanks to \eqref{4}:
\begin{align*}
A&= \left|\frac 1N \sum_{i=1}^N \mathbbm{1}_{\left\{x_i=x\right\}} \alpha^*_y(x_i, \Delta^i V^N (t,x_i, m_{\bm{x}}^{N,i}))
- \frac 1N \sum_{i=1}^N \mathbbm{1}_{\left\{x_i=x\right\}} \alpha^*_y(x_i, \Delta^x U(t,x_i,m^N_{\bm{x}}))\right|\\
&\leq C \frac 1N \sum_{i=1}^N \left| \Delta^i V^N (t,x_i, m_{\bm{x}}^{N,i}) - \Delta^x U(t,x_i,m^N_{\bm{x}})\right|\\
&\leq C \sup_{\bm{x} \in \Sigma^N}\frac 1N \sum_{i=1}^N \left|v^{N,i}(t,\bm{x}) - U(t, x_i, m^N_{\bm{x}})\right| \leq \frac CN.
\end{align*}
While $B$ converges to 0, uniformly over $t$, for the regularity  of $U$:
\begin{align*}
B&= \left|[m^N_{\bm{x}}]_x \alpha^*_y(x, \Delta^x U(t,x,m^N_{\bm{x}})) - m_x\alpha^*_y(x, \Delta^x U(t,x,m))\right|\\
&\leq |\alpha^*_y(x, \Delta^x U(t,x,m^N_{\bm{x}}))| |[m^N_{\bm{x}}]_x-m_x| +
C|m_x| |\Delta^x U(t,x,m^N_{\bm{x}}) -\Delta^x U(t,x,m)|\\
&\leq C |m^N_{\bm{x}}-m|,
\end{align*}
which tends to 0 by assumption.
\end{proof}

\section{The Master Equation: well-posedness and regularity}
\label{mastersection}

In this section we study the well-posedness of  Equation \eqref{eqn:M} under the assumptions of monotonicity and regularity for $F,G,H$ we already introduced  \textbf{(Mon)}, \textbf{(RegFG)}, \textbf{(RegH)}.
A preliminary remark is that, thanks to Proposition 1 in \cite{gomes}, if $H$ is differentiable (and this is indeed the case of our assumptions) then 
\begin{equation}
\alpha^*_x(y, p) = - \frac{\partial}{\partial p_x} H(y,p).
\end{equation}
For this reason, we will in the following use $\alpha^*$ interchangeably with $-D_p H$. 

\begin{thm}
Assume \textbf{(Mon)}, \textbf{(RegFG)} and \textbf{(RegH)}. Then there exists a unique classical solution to \eqref{eqn:M} in the sense of Definition \ref{def}. Moreover it is regular.
\label{Master}
\end{thm}

The proof  exploits the renowned method of characteristics, which consists in proving that 
\be
U(t_0, x, m_0) := u(t_0,x)
\label{defU}
\ee 
solves \eqref{eqn:M}, 
$u$ being the solution of the Mean Field Game system \eqref{eqn:MFG} with initial time $t_0$ and initial distribution $m_0$.
In order to perform the computations, we have to prove the regularity in $m$ of the function $U(t_0,x,m)$ defined above.
In particular, we have to show that $D^m U$ exists and  is bounded. For this, we follow the strategy shown in \cite{card} - which is developed in infinite dimension - adapting it to our discrete setting.
The idea consists in studying the well-posedness and regularity properties of the linearized version of the system \eqref{eqn:MFG}, whose solution will end up coinciding with $D^m U \cdot \mu_0$, for all possible directions $\mu_0 \in P_0(\Sigma)$. In the remaining part of this section, $C$ will denote any constant which does not depend on $t_0$, $m_0$, and is allowed to change from line to line.

\subsection{Estimates on the Mean Field Game system}

We start by proving the well-posedness of the system \eqref{eqn:MFG}
$$
\begin{cases}
-\frac{d}{dt} u(t,x) + H(x, \Delta^x u(t,x)) =  F(x, m(t)), \\
\frac{d}{dt} m_x(t) = \sum_y m_y(t) \alpha^{*}_x(y, \Delta^y u(t,y)),\\
u(T,x) = G(x, m(T)), \\
m_x(t_0) = m_{x,0},
\end{cases}
$$
 and a useful a priori estimate on its solution $(u,m)$.
The existence of solutions follows from a standard fixed point argument: see Proposition 4 of \cite{gomes}. Let us remark that any flow of measures $m$ lies in the space
$$\left\{ m\in C^0\left([t_0,T],P(\Sigma) \right): |m(t) -m(s)|\leq 2 \nu(\Xi)\sqrt{d} |t-s|\right\},$$
which is a compact and convex subset of the space of continuous functions, endowed with the uniform norm (Lemma 4 of \cite{alekos}).
On the other hand the uniqueness of solution, under our assumptions, is a consequence of the following a priori estimates.
Before stating the proposition, we recall the notation $||u|| := \sup_{t \in [t_0,T]} \max_{x \in \Sigma} |u(t,x)|$.

\begin{prop}
Assume \textbf{(Mon)}, \textbf{(RegFG)} and \textbf{(RegH)}. Let $(u_1,m_1)$ and $(u_2,m_2)$ be two solutions to
\eqref{eqn:MFG} with initial conditions $m_1(t_0) =m_0^1$ and $m_2(t_0) =m_0^2$. Then
\begin{align}
|| u_1 - u_2 || &\leq C |m_0^1 - m_0^2|,
\label{stimau}\\
|| m_1 - m_2|| &\leq C |m_0^1 - m_0^2|.
\label{stimam}
\end{align} 
\label{unimfg}
\end{prop}

\begin{proof}
Without loss of generality, let us set $t_0=0$. Let $u:=u_1 - u_2$ and $m:=m_1 - m_2$. The proof is carried out in three steps.

\emph{Step 1. Use of Monotonicity.}
The couple $(u,m)$ solves
\begin{equation}
\begin{cases}
-\frac{d}{dt} u(t,x) + H(x, \Delta^x u_1(t,x)) - H(x, \Delta^x u_2(t,x))=  F(x, m_1(t)) -F(x, m_2(t)) \\
\frac{d}{dt} m(t,x) = \sum_y \left[m_1(t,y) \alpha^{*}_x(y, \Delta^y u_1(t,y)) - m_2(t,y) \alpha^{*}_x(y, \Delta^y u_2(t,y))\right]\\
u(T,x) = G(x, m_1(T)) - G(x, m_2(T))\\
m(0,x) = m_0^1 - m_0^2.
\end{cases}
\label{um}
\end{equation}
Since $\frac{d}{dt} \sum_x m(x)u(x) = \sum_x m(x) \frac{du}{dt}(x) + \sum_x \frac{dm}{dt}(x) u(x)$,  integrating over $[0,T]$ we have
\begin{align*}
&\sum_x \left[m(T,x)u(T,x) - m(0,x)u(0,x) \right]\\
&= \int_0^T \sum_x \left[H(x, \Delta^x u_1) - H(x, \Delta^x u_2) 
 -F(x, m_1) +F(x, m_2)\right] (m_1(x) -m_2(x)) dt\\
&+\int_0^T \sum_x \sum_y \left[m_1(y) \alpha^{*}_x(y, \Delta^y u_1) - m_2(y) \alpha^{*}_x(y, \Delta^y u_2)\right]
 (u_1(x) - u_2(x)) dt.
\end{align*}
Using the fact that $\sum_x \alpha^*_x(y)=0$ and the initial-final data, we can rewrite
\begin{align*}
&\sum_x [G(x, m_1) -G(x, m_2)] (m_1(x) -m_2(x)) + \int_0^T \sum_x\left[F(x, m_1) -F(x, m_2)\right] (m_1(x) -m_2(x))dt \\
&= \sum_x (m_0^1(x) - m_0^2(x))(u_1(0,x)-u_2(0,x))\\
& \qquad +  \int_0^T \sum_x \left\{[H(x, \Delta^x u_1) - H(x, \Delta^x u_2)] (m_1(x) -m_2(x))\right.\\
&\qquad \left.+ \Delta^x u \cdot \left[
m_1(x) \alpha^{*}(x, \Delta^x u_1) - m_2(x) \alpha^{*}(x, \Delta^x u_2)\right]\right\}dt.
\end{align*}
We now apply the monotonicity of $F$ and $G$ in the first line and the uniform convexity of $H$ in the last two lines. In fact, recalling that 
$\alpha^*_y(x, p) = - \frac{\partial}{\partial p_y} H(x,p)$, by \textbf{(RegH)} we have that, for each $x$,
\begin{align*}
H(x, \Delta^x u_1) - H(x, \Delta^x u_2) - \Delta^x u \cdot \frac{\partial}{\partial p} H(x, \Delta^x u_1) &\leq 
-C^{-1} |\Delta^x u |^2
\\
H(x, \Delta^x u_2) - H(x, \Delta^x u_1) + \Delta^x u \cdot \frac{\partial}{\partial p} H(x, \Delta^x u_2) &\leq 
-C^{-1} |\Delta^x u |^2.
\end{align*}
Hence we obtain
\begin{equation}
\int_0^T \sum_x |\Delta^x u(x) |^2 (m_1(x) +m_2(x))dt \leq C (m_0^1 - m_0^2)\cdot(u_1(0)-u_2(0))
\label{monot}
\end{equation}

\emph{Step 2. Estimate on Kolmogorov-Fokker-Planck equation}.
Integrating the second  equation in \eqref{um} over $[0,t]$, we get
$$
m(t,x) = m(0,x) + \int_0^t \sum_y \left[m_1(s,y) \alpha^{*}_x(y, \Delta^y u_1(s,y)) - m_2(s,y) \alpha^{*}_x(y, \Delta^y u_2(s,y))\right]ds.
$$
The boundedness and Lipschitz continuity of the rates  give
$$
\max_x |m(t,x)| \leq C |m_0^1 - m_0^2| +C\int_0^t  \max_x |m(s,x)| ds+ C\int_0^t \sum_x|\Delta^x u(s,x) | m_1(s,x) ds
$$
and hence, by Gronwall's Lemma,
\begin{equation}
||m|| \leq C |m_0^1 - m_0^2| + C \int_0^T \sqrt{\sum_x|\Delta^x u(t,x) |^2 m_1(x)}dt.
\end{equation}
This, together with inequality \eqref{monot}, yields
\begin{equation}
||m|| \leq C (|m_0^1 - m_0^2| + |m_0^1 - m_0^2|^{1/2} ||u||^{1/2}).
\label{mcu}
\end{equation}

\emph{Step 3. Estimate on Hamilton-Jacobi-Bellman equation}.
Integrating the first  equation in \eqref{um} over $[t,T]$, we get
$$u(t,x) = G(x,m_1(T))-G(x,m_2(T)) +\int_t^T \left[F(x,m_1)-F(x,m_2) + H(x, \Delta^x u_2) - H(x, \Delta^x u_1)
\right]ds.$$
Using the Lipschitz continuity of $F,G,H$ and the bound $\max_x |\Delta^x u (x)| \leq C\max_x |u(x)|$ we obtain
$$\max_x |u(t,x)| \leq C |m_1(T) - m_2(T)| +C \int_t^T |m_1(s) - m_2(s)|ds + C \int_t^T \max_x |u(s,x)|ds$$
and then Gronwall's Lemma gives
\begin{equation}
||u||\leq C||m||.
\label{ucm}
\end{equation}

This bound \eqref{ucm} and estimate \eqref{mcu} yield claim \eqref{stimam}, using the convexity inequality $AB \leq \varepsilon A^2 + \frac{1}{4\varepsilon}B^2$ for $A, B > 0$.
Again \eqref{ucm} finally proves claim \eqref{stimau}.
\end{proof}

\subsection{Linearized MFG system}

For proving Theorem \ref{Master}, we introduce the linearized version of system \eqref{eqn:MFG} around its solutions and then prove that it provides the derivative of $u(t_0,x)$ with respect to the initial condition $m_0$.  

As a preliminary step, we study a related linear system of ODE's, which will come useful several times.
\begin{equation}
\label{eqn:z}
\begin{cases}
- \frac{d}{dt} z(t,x) - \alpha^*(x, \Delta^x u) \cdot \Delta^x z(t,x) = D^m F(x, m(t),1) \cdot \rho(t) + b(t,x)\\
\frac{d}{dt} \rho(t,x) = \sum_y \rho_y \alpha^*_x(y,\Delta^y u) + \sum_y m_y(t) D_p \alpha^*_x(y, \Delta^x u) \cdot \Delta^y z + c(t,x)\\
z(T,x) = D^m G(x,m(T),1) \cdot \rho(T) + z_T(x)\\
\rho(t_0, \cdot) = \rho_0,
\end{cases}
\end{equation}
The unknowns are $z$ and $\rho$, while $b,c,z_T,\rho_0$ are given measurable functions,  with $c(t)\in P_0(\Sigma)$, and $(u,m)$ is the solution to \eqref{eqn:MFG}.
We state an immediate but useful estimate regarding the first of the two equations in \eqref{eqn:z}. 
\begin{lem}
\label{prop}
If  \textbf{(RegFG)} holds then the equation
\begin{equation}
\label{eqn:lin_hjb}
\begin{cases}
- \frac{d}{dt} z(t,x) - \alpha^*(x, \Delta^x u) \cdot \Delta^x z(t,x) = D^m F(x, m(t),1) \cdot \rho(t) + b(t,x)\\
z(T,x) =  D^m G(x,m(T),1) \cdot \rho(T) + z_T(x) 
\end{cases}
\end{equation}
has a unique solution for each final condition $z_T(x)$ and satisfies 
\begin{equation}
\label{eqn:estimate}
||z|| \leq C \left[ \max_x |z_T(x)| + ||\rho|| + ||b||\right].
\end{equation}
\end{lem}

\begin{proof}
The well-posedness of the equation is immediate from classical ODE's theory.
Integrating over the time interval $[t,T]$ and using that 
\begin{equation*}
\alpha^*(x, \Delta^x u) \cdot \Delta^x z(t,x) = \sum_y \alpha^*_y(x, \Delta^x u) z_y(t),
\end{equation*}
we find
\begin{align*}
z(t,x) - z(T,x) - \int_t^T \sum_y \alpha^*_y(x, \Delta^x u) z_y(s)ds = \int_t^T D^m F \cdot \rho(s)ds + \int_t^T b(s,x) ds.
\end{align*}
Substituting the expression for $z(T,x)$, 
and using the bound on the control and on the derivatives of $F$ and $G$ we can estimate
\begin{align*}
\max_x |z(t,x)| & \leq \max_x |z_T(x)| +  C\max_x |\rho(T,x)|\\
& + C \int_t^T \max_x |z(s,x)|ds + C \int_t^T \max_x |\rho(s,x)|ds + \int_t^T \max_x |b(s,x)| ds
\end{align*}
and thus, applying Gronwall's Lemma and taking the supremum on $t$, we  get \eqref{eqn:estimate}.
\end{proof}

In the next result we prove the well-posedness of system \eqref{eqn:z} together with useful a priori estimates on its solution. 
\begin{prop}
Assume \textbf{(RegH)}, \textbf{(Mon)} and \textbf{(RegFG)}. Then for any (measurable) $b,c,z_T$, the linear system \eqref{eqn:z} has a unique solution $(z,\rho) \in C^1([0,T]; \mathbb{R}^d \times P_0(\Sigma))$. Moreover it satisfies
\begin{align}
\label{eqn:apriori}
||z|| & \leq C(|z_T| + ||b|| + ||c|| + |\rho_0|)  \\
\label{eqn:apriori2}
||\rho|| & \leq C(|z_T| + ||b|| + ||c|| + |\rho_0|).
\end{align}
\label{proppriori}
\end{prop}

\begin{proof}
Without loss of generality we assume $t_0=0$.
We use a fixed-point argument to prove the existence of a solution to \eqref{eqn:z}. Uniqueness will be then implied by  estimates \eqref{eqn:apriori} and \eqref{eqn:apriori2}, thanks to the linearity of the system.

We define the map $\Phi : C^0\left([0,T] ; P_0(\Sigma)\right) \to  C^0\left([0,T] ; P_0(\Sigma)\right)$ as follows: for a fixed $\rho \in C^0\left([0,T] ; P_0(\Sigma)\right)$ we consider the solution $z = z(\rho)$ to equation \eqref{eqn:lin_hjb}, and define $\Phi(\rho)$ to be the solution of the second equation in \eqref{eqn:z} with $z = z(\rho)$.
In order to prove the existence of a fixed point of $\Phi$, which is clearly a solution to \eqref{eqn:z}, we apply Leray-Schauder Fixed Point Theorem.
We remark the fact that more standard fixed point theorems are not applicable to this situation since we cannot assume that $\rho$ belongs to a compact subspace of $C^0\left([0,T] ; P_0(\Sigma)\right)$, since $P_0(\Sigma)$ is not compact.  
First of all, we note that $C^0\left([0,T] ; P_0(\Sigma)\right)$ is convex and that the map $\Phi$ is trivially continuous, because of the linearity of the system.
Moreover, using the equation for $\rho$ in system \eqref{eqn:z}, it is easy to see that $\Phi$ is a compact map, i.e. it sends bounded sets of $C^0\left([0,T] ; P_0(\Sigma)\right)$ into bounded sets of $C^1\left([0,T] ; P_0(\Sigma)\right)$.
Thus, to apply Leray-Schauder Theorem it remains to prove that the set $\left\{\rho \ : \ \rho = \lambda \Phi(\rho) \text{ for some } \lambda \in [0,1]\right\}$ is bounded in $C^0\left([0,T] ; P_0(\Sigma)\right)$.

Let us fix a $\rho$ such that $\rho = \lambda \Phi(\rho)$. Then the couple $(z, \rho)$ solves
\begin{equation*}
\begin{cases}
- \frac{d}{dt} z(t,x) - \alpha^*(x, \Delta^x u) \cdot \Delta^x z(t,x) = \lambda\left(D^m F(x, m(t),1) \cdot \rho(t) + b(t,x)\right)\\
\frac{d}{dt} \rho(t,x) = \sum_y \rho_y \alpha^*_x(y,\Delta^y u) + \lambda\left(\sum_y m_y(t) D_p \alpha^*_x(y, \Delta^x u) \cdot \Delta^y z + c(t,x)\right)\\
z(T,x) = \lambda\left(D^m G(x,m(T),1) \cdot \rho(T) + z_T(x)\right)\\
\rho(t_0, \cdot) = \lambda \rho_0.
\end{cases}
\end{equation*}
First, we note that we can restrict to $\lambda > 0$, since otherwise $\rho = 0$. Therefore, we can use the equations (for brevity we omit the dependence of $\alpha^*$ on the second variable) to get 
\begin{align*}
\frac{d}{dt} \sum_x z(t,x) \rho_x(t)  =& -\lambda \sum_x \rho(t,x)[D^m F(x,m(t), 1) \cdot \rho(t) + b(t,x)] \\ 
&-\sum_{x,y} \rho_x(t) \alpha^*_y(x) [z(t,y)-z(t,x)] + \sum_{x,y} \rho_y(t) \alpha^*_x(y) z(t,x)\\
&+ \lambda\sum_{x,y}m_y z(t,x) D_p \alpha_x^*(y)  \cdot \Delta^y z  + \lambda \sum_x c(t,x) z(t,x).
\end{align*}
The second line is 0, using the fact that $\sum_x \rho_x(t) =0$ and changing $x$ and $y$ in the second double sum.
Integrating over $[0,T]$ and using the expression for $z(T,x)$ we obtain
\begin{align*}
\lambda  \sum_x \rho_x(T)&\left[D^m G(x,m(T), 1) \cdot \rho(T) + z_T(x)   \right] - \lambda z(0) \cdot \rho_0\\
= &- \lambda \int_0^T  \sum_x \rho_x(t)[D^m F(x,m(t), 1) \cdot \rho(t) + b(t,x)]dt \\
& + \lambda \int_0^T \sum_{x,y} m_y D_p \alpha_x^*(y)  \cdot \Delta^y z(z(t,x) - z(t,y))dt \\
& + \lambda \int_0^T \sum_x c(t,x) z(t,x)dt - \lambda \int_0^T \rho(t,x) D^m G(x,m(T), 1)\cdot \rho(T)dt, 
\end{align*}
where in the second term of the sum we have also used that $\sum_{x,y} [m_y D_p\alpha^*_x(y) \cdot \Delta^y z] z(t,y) = 0$.

Dividing by $\lambda >0$ and bringing the terms with $F$ and $G$ on the left hand side, together with the term in $m$ and $D_p \alpha^*$, we can rewrite 
\begin{align*}
-& \int_0^T \sum_{x,y}  m_y \Delta^y z D_p\alpha^*_x(y) \cdot \Delta^y z dt + \int_0^T  \sum_x \rho(t,x)[D^m F(x,m(t), 1) \cdot \rho(t)]dt \\
 +&\sum_x \rho(T,x) D^m G(x,m(T), 1)\cdot \rho(T) \\
& \qquad = - \sum_x z_T(x) \rho(T,x) + \sum_x z(0,x) \rho_0(x) 
 - \int_0^T \sum_x \rho(t,x)b(t,x) dt +  \int_0^T \sum_x c(t,x) z(t,x)dt.
\end{align*}
We observe that, by \textbf{(Mon)} and \textbf{(RegFG)}, we have 
\begin{align}
 \sum_x \rho(t,x)[D^m F(x,m(t), 1) \cdot \rho(t)]&\geq 0, \\
 \sum_x \rho(T,x) [D^m G(x,m(T), 1)\cdot \rho(T)]&\geq 0.
\end{align}
Furthermore assumption \eqref{bound} yields 
\begin{equation*}
-\int_0^T \sum_{x,y}  m_y \Delta^y z D_p\alpha^*_x(y) \cdot \Delta^y z dt \geq C^{-1}\int_0^T \sum_x m_x |\Delta^x z|^2dt,
\end{equation*} 
so that we can estimate the previous equality by
\begin{align}
C^{-1}\int_0^T & \sum_x m_x |\Delta^x z|^2 dt \leq  |z_T \cdot \rho(T)| + |z(0) \cdot \rho_0| +\int_0^T \left|c(t) \cdot z(t) \right|dt
 + \int_0^T \left|\rho(t) \cdot b(t)\right|dt \nonumber\\
& \leq  |z_T| |\rho(T)| + |z(0)| |\rho_0|+ \int_0^T |c(t)| |z(t)|dt + \int_0^T |\rho(t)| |b(t)| dt
\label{sti}
\end{align}

On the other hand, by the equation for $\rho$ we have
\begin{align*}
\rho(t,x) = \rho_0(x) +\int_0^t \sum_y \rho(s,y) \alpha_x^*(y)ds + \int_0^t \left[\sum_y m_y D_p \alpha^*_x(y) \cdot \Delta^y z + c(x)\right]ds,
\end{align*}
and thus
\begin{align*}
|\rho(t,x)| & \leq  |\rho_0(x)| + M\int_0^t \sum_y |\rho_y|ds + C\int_0^t \left[\sum_y m_y |\Delta^y z| + |c(x)|\right]ds,
\end{align*}
so that, by Gronwall's Lemma and taking the sum for $x\in\Sigma$ and the sup over $t \in [0,T]$, 
\begin{align*}
||\rho|| & \leq C|\rho_0| + C \int_0^T \sum_x \sqrt{m_x} \sqrt{m_x} |\Delta^x z|dt + C ||c||\\
& \leq C|\rho_0| + C \int_0^T \sqrt{\sum_x\left(\sqrt{m_x}\right)^2} \sqrt{\sum_x m_x |\Delta^x z|^2}dt + C||c|| \\
& = C|\rho_0| +C \int_0^T \sqrt{\sum_x m_x |\Delta^x z|^2}dt + C||c|| \\
& \leq C|\rho_0| +C \sqrt{\int_0^T \sum_x m_x |\Delta^x z|^2dt} + C||c||.
\end{align*}

Now, we use  estimate \eqref{sti} on $\int_0^T \sum_x m_x |\Delta^x z|^2$ that we found above to get
\begin{align*}
||\rho|| & \leq C ||c|| +C|\rho_0| + C\left( |\rho_0||z(0)| + |z_T| |\rho(T)| + \int_0^T |c(t)| |z(t)| + \int_0^T |\rho(t)| |b(t)| \right)^{\frac{1}{2}}\\
& \leq C ||c|| +C|\rho_0| + C\left(|z(0)|^{1/2} |\rho_0|^{1/2} + |z_T|^{1/2} |\rho(T)|^{1/2}| + ||c||^{1/2} ||z||^{1/2}  + ||\rho||^{1/2} ||b||^{1/2} \right).
\end{align*}
We  further estimate the right hand side using  bound \eqref{eqn:estimate}:
\begin{align*}
||\rho||  &\leq C (||c||+ |\rho_0|) \\
&+ C\left[ |z_T|^{1/2} |\rho(T)|^{1/2} + (||c||^{1/2}+|\rho_0|^{1/2})(|z_T|^{1/2} + ||\rho||^{1/2} + ||b||^{1/2}) 
+ ||\rho||^{1/2} ||b||^{1/2} \right].
\end{align*}
Using the inequality $AB \leq \varepsilon A^2 + \frac{1}{4\varepsilon}B^2$ for $A, B > 0$, we obtain
\begin{equation*}
||\rho|| \leq C(||c|| + |z_T|+ ||b|| +|\rho_0|) +\frac12 ||\rho||,
\end{equation*} 
which implies \eqref{eqn:apriori2}.  Then \eqref{eqn:apriori} follows  from \eqref{eqn:estimate}.
\end{proof}

Given the solution $(u,m)$ to system \eqref{eqn:MFG}, with initial condition $m_0$ for $m$ and final condition $G$ for $u$, we introduce the linearized system:
\begin{equation}
\tag{LIN}
\label{eqn:LIN}
\begin{cases}
-\frac{d}{dt} v(t,x) - \alpha^*(x, \Delta^x u(t,x)) \cdot \Delta^x v(t,x) = D^mF(x, m(t), 1) \cdot \mu(t)\\
\frac{d}{dt} \mu_x(t) = \sum_y \mu_y(t) \alpha^*_x(y, \Delta^y u(t,y)) + \sum_y m_y D_p \alpha^*_x(y,\Delta^y u) \cdot \Delta^y v(t,x) \\
v(T,x) = D^m G(x,m(T),1) \cdot \mu(T)\\
\mu(t_0) = \mu_0 \in P_0(\Sigma).
\end{cases}
\end{equation}
Note that in the right hand side of the first equation $D^m F(x,m(t),1) \cdot \mu(t) = D^m F(x, m(t), y) \cdot \mu(t)$ for every $y \in \Sigma$, using identity \eqref{eqn:identity} and the fact that $\mu(t) \in P_0(\Sigma)$ for every $t$ (i.e. identity \eqref{eqn:identity2}).
For this reason we just fixed the choice to $D^m F(x,m(t), 1)$ and $D^m G(x, m(T),1)$ in system \eqref{eqn:LIN}.

The existence and uniqueness of a solution $(v, \mu) \in C^1 ([0,T]; \mathbb{R}^d \times P_0(\Sigma))$ is ensured by Proposition
\ref{proppriori}.
The aim is to show that the solution $(v, \mu)$ to system \eqref{eqn:LIN} satisfies  
\begin{equation}
v(t_0,x) = D^m U(t_0, x, m_0, 1) \cdot \mu_0.
\label{deri}
\end{equation}
This proves that the solution $U$ defined via \eqref{defU} is differentiable with respect to $m_0$ in any direction $\mu_0$, with derivative given by \eqref{deri}, and also that $D^m U$ is continuous in $m$.
 Equality \eqref{deri} is implied by the following 

\begin{thm}
\label{reg}
Assume \textbf{(RegH)}, \textbf{(Mon)} and \textbf{(RegFG)}.
Let $(u,m)$ and $(\hat{u}, \hat{m})$ be the solutions to \eqref{eqn:MFG} respectively starting from $(t_0, m_0)$ and $(t_0, \hat{m}_0)$. Let $(v, \mu)$ be the solution to \eqref{eqn:LIN} starting from $(t_0, \mu_0)$, with $\mu_0 := \hat{m}_0 - m_0$. Then
\begin{equation}
||\hat{u} - u - v|| + ||\hat{m}- m- \mu||  \leq C |m_0 - \hat{m}_0|^2.
\label{bound2}
\end{equation}
\end{thm}

\begin{proof}

Set $z:= \hat{u} - u -v$ and $\rho := \hat{m} - m - \mu$, they solve \eqref{eqn:z}
$$
\begin{cases}
- \frac{d}{dt} z(t,x) - \alpha^*(x, \Delta^x u) \cdot \Delta^x z(t,x) = D^m F(x, m(t),1) \cdot \rho(t) + b(t,x)\\
\frac{d}{dt} \rho(t,x) = \sum_y \rho_y \alpha^*_x(y,\Delta^y u) + \sum_y m_y(t) D_p \alpha^*_x(y, \Delta^x u) \cdot \Delta^y z + c(t,x)\\
z(T,x) = D^m G(x,m(T),1) \cdot \rho(T) + z_T(x)\\
\rho(t_0, \cdot) = 0,
\end{cases}
$$
with 
\begin{align*}
b(t,x) & := A(t,x) + B(t,x)\\
A(t,x)& := - \int_0^1 \left[D_p H(x, \Delta^x u + s(\Delta^x \hat{u} - \Delta^x u)) - D_p H(x, \Delta^x u)\right] \cdot (\Delta^x \hat{u} - \Delta^x u) ds \\
B(t,x) & := \int_0^1 \left[D^m F(x, m+ s(\hat{m} - m),1) - D^mF(x,m,1)\right] \cdot (\hat{m} - m) ds\\
c(t,x) & := \sum_y (\hat{m}_y - m_y) D_p \alpha^*_x(y, \Delta^y u) \cdot (\Delta^y \hat{u} - \Delta^y u) \\
& + \sum_y \hat{m}_y \int_0^1 \left[D_p \alpha^*_x(y, \Delta^x u + s(\Delta^x \hat{u} - \Delta^x u)) - D_p \alpha^*_x(y, \Delta^y u)\right] \cdot(\Delta^y \hat{u} - \Delta^y u) ds\\
z_T(x) & := \int_0^1 \left[D^m G(x,m(T) + s(\hat{m}(T) - m(T)),1) - D^mG(x,m(t),1)\right] \cdot (\hat{m}(T) - m(T))ds.
\end{align*}

Using the assumptions, namely the Lipschitz continuity of $D_p H, \ D^2_{pp}H, \ D^m F$ and $D^m G$, and the bound 
$\max_x |\Delta^x u| \leq C |u|$, we estimate
\begin{align*}
||b||&\leq ||A|| + ||B||\\
||A|| &\leq C || \hat{u} - u||^2\\
||B|| &\leq C || \hat{m} - m ||^2\\
|z_T|&\leq C |\hat{m}(T) - m(T)|^2\\
||c||&\leq C || \hat{m} - m ||\cdot || \hat{u} - u|| + C || \hat{u} - u||^2.
\end{align*}
Applying \eqref{eqn:apriori} and \eqref{eqn:apriori2} to the above system and then \eqref{stimau} and \eqref{stimam}, we obtain
\begin{align*}
|| z|| +||\rho|| &\leq C (|z_T| + ||b|| + ||c||)\\ 
&\leq C \left(|| \hat{u} - u||^2 + || \hat{m} - m ||^2 + || \hat{m} - m ||\cdot || \hat{u} - u||\right) \\
&\leq C |m_0 - \hat{m}_0|^2. 
\end{align*}
\end{proof}

\subsection{Proof of Theorem \ref{Master}}
We are finally in the position to prove the main theorem of this section.

\subsubsection{Existence}
Let $U$ be the function defined by \eqref{defU}, i.e. $U(t_0, x,m_0) := u(t_0,m_0)$. We have shown in the above Theorem \ref{reg}
that $U$ is $C^1$ in $m$, while the fact that it is $C^1$ in $t$ is clear.  We compute the limit, as $h$ tends to 0, of
\begin{align}
\label{eqn:formal}
& \frac{U(t_0+h, x, m_0) - U(t_0, x,m_0)}{h} \nonumber \\ 
& = \frac{U(t_0+h, x,m_0) - U(t_0+h, x,m(t_0+h))}{h} + \frac{U(t_0+h,x,m(t_0+h)) - U(t_0,x,m_0)}{h}.
\end{align}
For the first term,  we have, for any $y\in\Sigma$,
\begin{align*}
U(t_0 + h, x , m(t_0 +h)) & - U(t_0 + h, x,m(t_0)) \\
& = [m_s := m(t_0) + s(m(t_0+h) - m(t_0))]\\
& = \int_0^1  \frac{\partial}{\partial({m(t_0 + h) - m(t_0)})}U(t_0 + h, x ,m_s, y) ds\\ 
& = \int_0^1 D^m U(t_0+h,x,m_s, y) \cdot (m(t_0 + h) - m(t_0))ds\\
& = \int_0^1 ds \int_{t_0}^{t_0+h}  D^m U(t_0+h,x,m_s, y) \cdot \left(\sum_{k=1}^d m_k(t) \alpha^*(k, \Delta^k u(t))\right)dt\\
& = \int_0^1 ds \int_{t_0}^{t_0+h}  \sum_{z=1}^d \sum_{k=1}^d m_k(t) \left[D^m U(t_0+h,x,m_s, y)\right]_z \alpha^{*}_z(k, \Delta^k u(t))dt.
\end{align*}
Using identity \eqref{eqn:identity}, we obtain
\begin{align*}
U(t_0 + h, x , m(t_0 +h)) & - U(t_0 + h, x,m(t_0)) \\
& = \int_0^1 ds \int_{t_0}^{t_0+h}  \sum_{z=1}^d \sum_{k=1}^d m_k(t) \left[D^m U(t_0+h,x,m_s, k)\right]_z \alpha^{*}_z(k, \Delta^k u(t))dt \\
& + \int_0^1 ds \int_{t_0}^{t_0+h}  \sum_{z=1}^d \sum_{k=1}^d m_k(t) \left[D^m U(t_0+h,x,m_s, y)\right]_k \alpha^{*}_z(k, \Delta^k u(t))dt\\
& = \int_0^1 ds \int_{t_0}^{t_0+h}  \sum_{z=1}^d \sum_{k=1}^d m_k(t) \left[D^m U(t_0+h,x,m_s, k)\right]_z \alpha^{*}_z(k, \Delta^k u(t))dt,
\end{align*}
where the last equality follows from 
\begin{align*}
\sum_{z=1}^d & \sum_{k=1}^d m_k(t) \left[D^m U(t_0+h,x,m_s, y)\right]_k \alpha^{*}_z(k, \Delta^k u(t))\\
& = \sum_{k=1}^d m_k(t) \left[D^m U(t_0+h,x,m_s, y)\right]_k \sum_{z=1}^d \alpha^{*}_z(k, \Delta^k u(t)) = 0,
\end{align*}
since $\sum_{z=1}^d \alpha^{*}_z = 0$, as $\alpha^{*}_k(k) = - \sum_{z \neq k} \alpha^{*}_z(k)$.

Summarizing, we have found that,
\begin{align*}
U(t_0+h, x, m(t_0+h)) & - U(t_0+h, x,m(t_0)) \\
& = \int_0^1 ds \int_{t_0}^{t_0+h} \!\!\! dt \int_{\Sigma} D^m U(t_0+h,x,m_s, y) \cdot \alpha^*(y, \Delta^y u(t)) m(t)(dy) .
\end{align*}
Dividing by $h$ and letting $h \to 0$, we get
\begin{align*}
\lim_{h \to 0}& \frac{U(t_0+h, x, m(t_0+h))  - U(t_0+h, x,m(t_0))}{h} \\
& = \int_{\Sigma} D^m U(t_0,x,m_0, y) \cdot \alpha^*(y, \Delta^y u(t_0)) dm_0(y) \\
& =  \int_{\Sigma} D^m U(t_0,x,m_0, y) \cdot \alpha^*(y, \Delta^x U(t_0,y,m_0)) dm_0(y),
\end{align*}
using the continuity of $D^m U$ in time and dominate convergence to take the limit inside the integral in $ds$.

The second term in \eqref{eqn:formal}, for $h > 0$, is instead
\begin{equation*}
U(t_0+h, x,m(t_0+h)) - U(t_0,x,m_0) = u_x(t_0+h) - u_x(t_0) = h \frac{d}{dt}u_x(t_0) + o(h),
\end{equation*}
and thus
\begin{equation*}
\lim_{h \to 0^+}\frac{U(t_0+h, x,m(t_0+h)) - U(t_0,x,m_0)}{h} = \frac{d}{dt}u_x(t_0).
\end{equation*}
Finally, we can rewrite \eqref{eqn:formal}, after taking the limit $h \to 0$, to obtain
\begin{align*}
\partial_t U(t_0, x,m_0) & = -\int_{\Sigma} D^m U(t_0,x,m_0, y) \cdot \alpha^*(y, \Delta^x U(t_0,y,m_0)) dm_0(y) \\
& + \frac{d}{dt}u_x(t_0) = [\text{using the equation for $u$}]\\
& =  -\int_{\Sigma} D^m U(t_0,x,m_0, y) \cdot \alpha^*(y, \Delta^x U(t_0,y,m_0)) dm_0(y) \\
& + H(x,\Delta^xU(t_0,x,m_0)) - F(x,m_0),
\end{align*}
and thus
\begin{equation*}
-\partial_t U(t_0,x,m_0) + H(x, \Delta^x U(t_0,x,m_0)) - \int_{\Sigma} D^mU(t_0,x,m_0,y) \cdot \alpha^*(y, \Delta^y U) dm_0(y) = F(x,m_0),
\end{equation*}
which is exactly \eqref{eqn:M} computed in $(t_0,m_0)$.

\subsubsection{Uniqueness}
Let us consider another solution $V$ of \eqref{eqn:M}. Since $||D^m V|| \leq C$, we know that $V$ is Lipschitz with respect to $m$, and so is $\Delta^x V$.
From this remark and the Lipschitz continuity of $\alpha^*$ with respect to $p$
, it follows that the equation 
\begin{equation*}
\begin{cases}
\frac{d}{dt} \tilde{m}(t) = \sum_y \tilde{m}_y(t) \alpha^{*}(y, \Delta^y V(t,y,\tilde{m}(t)))\\
\tilde{m}(t_0) = m_0 
\end{cases}
\end{equation*} 
admits a unique solution in $[t_0,T]$. 

If we now set $\tilde{u}(t,x) := V(t,x,\tilde{m}(t))$, we can compute (using for e.g. $D^m V(\cdot,\cdot,\cdot,1)$) 
\begin{align*}
\frac{d}{dt} \tilde{u}(t,x) & = \partial_t V(t,x,\tilde{m}(t)) + D^m V(t,x,\tilde{m}(t), 1) \cdot \frac{d}{dt} \tilde{m}(t) \\
& = [\text{using the equation for $\tilde{m}$}] \\
& = \partial_t V(t,x,\tilde{m}(t)) + D^m V(t,x,\tilde{m}(t), 1) \cdot \left(\sum_y \tilde{m}_y(t) \alpha^{*}(y, \Delta^y V(t,y,\tilde{m}(t))) \right)\\
& = [\text{using identity \eqref{eqn:identity} on $D^m V(\cdot,\cdot,\cdot,1)$]}\\
& = \partial_t V(t,x,\tilde{m}(t)) + \int_{\Sigma}D^m V(t,x,\tilde{m}(t), y) \cdot \alpha^{*}(y, \Delta^y V(t,y,\tilde{m}(t))) 
 \tilde{m}(t)(dy)\\
& = [\text{using the equation for $V$}]\\
& = H(x,\Delta^x V(t,x,\tilde{m}(t))) - F(x,\tilde{m}) = H(x, \Delta^x \tilde{u}(t,x)) - F(x, \tilde{m}(t)), 
\end{align*}
and thus the pair $(\tilde{u}(t), \tilde{m}(t))$ satisfies
\begin{equation*}
\begin{cases}
-\frac{d}{dt} \tilde{u}(t,x) + H(x, \Delta^x \tilde{u}(t,x)) =  F(x, \tilde{m}(t)),\\
\frac{d}{dt} \tilde{m}_x(t) = \sum_j \tilde{m}_y(t)\alpha^{*}_x(y,\Delta^y \tilde{u}(t,y)),\\
\tilde{u}(T,x) = V(T,x,\tilde{m}(T)) = G(x, \tilde{m}(T)),\\
\tilde{m}(t_0) = m_0.
\end{cases}
\end{equation*}
Namely, $(\tilde{u}, \tilde{m})$ solves the system \eqref{eqn:MFG}, whose solution is unique thanks to Proposition \ref{unimfg}, so that we can conclude $V(t_0, x, m_0) = U(t_0, x, m_0)$ for each $(t_0,x,m_0)$, and thus the uniqueness of solutions to \eqref{eqn:M} follows.

\subsubsection{Regularity}

It remains to prove that the unique classical solution defined via \eqref{defU} is regular, in the sense of Definition \ref{def}, i.e. that
$D^m U$ is Lipschitz continuous with respect to $m$, uniformly in $t,x$.

So let $(u_1,m_1)$ and $(u_2,m_2)$ be two solution to \eqref{eqn:MFG} with initial conditions $m_1(t_0)=m_0^1$ and $m_2(t_0)=m_0^2$, respectively.
Let also $(v_1, \mu_1)$ and $(v_2,\mu_2)$ be the associated solutions to \eqref{eqn:LIN} with $\mu_1(t_0) = \mu_2(t_0) =\mu_0$.
Recall from equation \eqref{deri} that $v_1(t_0,x) = D^m U(t_0, x, m_0^1, 1) \cdot \mu_0$ and 
$v_2(t_0,x) = D^m U(t_0, x, m_0^2, 1) \cdot \mu_0$, thus we have to estimate the norm $||v_1 -v_2||$.

Set $z:=v_1-v_2$ and $\rho:= \mu_1-\mu_2$. They solve the linear system \eqref{eqn:z} with $\rho_0=0$ and 
\begin{align*}
b(t,x) & := \left[D^m F(x,m_1,1)- D^m F(x,m_2,1) \right] \cdot\mu_2 
+\left[\alpha^*(x, \Delta^x u_1) - \alpha^*(x, \Delta^x u_2)\right]\cdot  \Delta^x v_2
\\
c(t,x) & := \sum_y \mu_{2,y} \left[\alpha^*_x(y, \Delta^y u_1) - \alpha^*_x(y, \Delta^y u_2)\right]
 \\
& \quad + \sum_y \left[m_{1,y} D_p \alpha^*_x(y, \Delta^y u_1) - m_{2,y} D_p \alpha^*_x(y, \Delta^y u_2) \right] \cdot \Delta^x v_2
\\
z_T(x) & := \left[D^m G(x,m_1(T),1) - D^m G(x,m_2(t),1)\right] \cdot \mu_2.
\end{align*}
Using the Lipschitz continuity of $D_pH$, $D^2_{pp}H$, $D^m F$ and $D^m G$, 
applying the bounds \eqref{eqn:apriori} to $v_2$ and \eqref{eqn:apriori2} 
to $\mu_2$ and also \eqref{stimau} and \eqref{stimam}, we estimate
\begin{align*}
||b|| &\leq C ||m_1-m_2||\cdot ||\mu_2|| + C ||u_1-u_2||\cdot ||v_2|| \leq C|m_0^1-m_0^2| \cdot|\mu_0|\\
||c|| &\leq C ||u_1-u_2|| \cdot||\mu_2|| + C ||m_1-m_2||\cdot ||v_2|| + C ||u_1-u_2||\cdot ||v_2|| \leq C|m_0^1-m_0^2| \cdot|\mu_0|\\
|z_T| &\leq C ||m_1-m_2||\cdot ||\mu_2|| \leq C|m_0^1-m_0^2|\cdot |\mu_0|.
\end{align*}
Then \eqref{eqn:apriori} gives
$$||z||\leq C ( ||b|| + ||c|| + |z_T|) \leq C  |m_0^1-m_0^2|\cdot |\mu_0|,$$
which, since $z(t_0,x) = \left(D^m U(t_0, x, m_0^1, 1) -D^m U(t_0, x, m_0^2, 1)\right)\cdot \mu_0$, yields
\begin{align*}
\max_x |D^m U(t_0,& x, m_0^1, 1) -D^m U(t_0, x, m_0^2, 1)| \\
&\leq C \max_x \sup_{\mu_0 \in P_0(\Sigma)}
\frac{\left|\left(D^m U(t_0, x, m_0^1, 1) -D^m U(t_0, x, m_0^2, 1)\right)\cdot \mu_0\right|}{|\mu_0|} \\
&\leq C |m_0^1-m_0^2|. 
\end{align*}

\section{Conclusions}
\label{concl}
Let us summarize the results we have obtained. The two set of assumptions are given in Section 2.2 and verified in Example 2.1.
\begin{enumerate}
	\item If \textbf{(H1)} holds and there exists a regular solution $U$ to the Master Equation \eqref{eqn:M},
	in the sense of Definition \ref{def}, then the value functions of the $N$-player game converge to $U$ (Theorem \ref{fund}) and the optimal trajectories \eqref{eqn:13} satisfy a propagation of chaos property, i.e they converge to the limiting i.i.d solution to \eqref{limit} 
	(Theorem \ref{chaos});
	\item Under the assumptions required for convergence, the empirical measures processes \eqref{mn} associated with the optimal trajectories satisfy a Central limit Theorem (Theorem  \ref{clt}) and a Large Deviation Principle with rate function $I$ in \eqref{3.8} (Theorem \ref{large});
	\item Assuming \textbf{(RegH)}, \textbf{(Mon)} and \textbf{(RegFG)}, there exists a unique classical solution to \eqref{eqn:M} and it is also regular in the sense of Definition \ref{def}.
\end{enumerate}

\end{document}